\documentclass[pdftex]{amsart}
\usepackage{amsmath, amssymb, stmaryrd, amsrefs}
\usepackage[all]{xy}
\allowdisplaybreaks
\usepackage{amsthm}
\theoremstyle{definition}
\newtheorem{definition}{Definition}[section]
\theoremstyle{plain}
\newtheorem{theorem}[definition]{Theorem}
\newtheorem{proposition}[definition]{Proposition}
\newtheorem{lemma}[definition]{Lemma}
\newtheorem{corollary}[definition]{Corollary}
\newtheorem{fact}[definition]{Fact}
\theoremstyle{remark}
\newtheorem*{remark}{Remark}
\newtheorem{example}[definition]{Example}
\usepackage[svgnames]{xcolor}
\usepackage{hyperref}
\hypersetup{
  colorlinks=true,
  linkcolor=Blue,
  citecolor=Green,
  pdftitle={},
  pdfpagemode=FullScreen
}
\newcommand\e{{\boldsymbol e}}
\renewcommand\v{{\boldsymbol v}}
\renewcommand\H{{\mathcal H}}
\newcommand\I{{\mathcal I}}
\renewcommand\L{{\mathcal L}}
\newcommand\K{{\mathcal K}}
\renewcommand\O{{\mathcal O}}
\newcommand\U{{\mathcal U}}
\newcommand\N{{\mathbb N}}
\newcommand\Z{{\mathbb Z}}
\newcommand\C{{\mathbb C}}

\newcommand\X{{\mathbb X}}
\newcommand\id{{\mathrm{id}}}
\newcommand\inv[1]{#1^{-1}}

\newcommand\clspan{\mathop{\overline{\mathrm{span}}}}
\newcommand\supp{\mathop{\mathsf{supp}}}
\newcommand\defiff{\overset{\mathrm{def}}{\iff}}
\newcommand\rel[1]{\llbracket#1\rrbracket}
\newcommand\ent{\mathop{\mathrm{ent}}\nolimits}
\newcommand\stacksum[2]{\sum_{\substack{#1\\#2\rule{0em}{1.7ex}}}}
\newcommand\set[2]{\mathopen\lbrace#1\mathbin\vert#2\mathclose\rbrace}
\newcommand\SET[3]{\mathopen{#1\lbrace}#2\mathbin{#1\vert}#3\mathclose{#1\rbrace}}
\newcommand\inn[2]{\langle#1\vert#2\rangle}
\newcommand\INN[3]{\mathopen{#1\langle}#2\mathbin{#1\vert}#3\mathclose{#1\rangle}}
\newcommand\abs[1]{\vert#1\vert}

\newcommand\norm[1]{\Vert#1\Vert}
\newcommand\NORM[2]{\mathopen{#1\Vert}#2\mathclose{#1\Vert}}
\begin{document}
\title
[Cartan subalgebras of Kajiwara--Watatani algebras]
{Cartan subalgebras of C*-algebras associated\\with complex dynamical systems}
\author[K. Ito]{Kei Ito}
\address{Graduate School of Mathematical Sciences, The University of Tokyo, 3-8-1 Komaba, Tokyo, 153-8914, Japan}
\email{itokei@ms.u-tokyo.ac.jp}
\keywords{Cartan subalgebra, C*-algebra, complex dynamical system, branch point.}
\subjclass[2020]{46L55}
\maketitle

\begin{abstract}
Let $R$ be a rational function with degree $\geq 2$ and $X$ be its Julia set, its Fatou set, or the Riemann sphere.
Suppose that $X$ is not empty.
We can regard $R$ as a continuous map from $X$ onto itself.
Kajiwara and Watatani showed that in the case that $X$ is the Julia set, $C_0(X)$ is a maximal abelian subalgebra of $\O_R(X)$, where $\O_R(X)$ denotes the C*-algebra associated with the dynamical system~$(X,R)$ introduced by them.
In this paper, we develop their result and give the equivalent condition for~$C_0(X)$ to be a Cartan subalgebra of~$\O_R(X)$.
\end{abstract}

\section{Introduction}
\subsection{Historical background}

In this paper, we study a Cartan subalgebra of a C*-algebra associated with a complex dynamical system.

A C*-algebra associated with a complex dynamical system was introduced by Kajiwara and Watatani~\cite{KW05}.
Their aim was to modify the definition of Deaconu and Muhly's algebra.

Deaconu and Muhly~\cite{DM} introduced a C*-algebra associated with a branched covering by using Renault's groupoid method~\cite{R}.
For a branched covering $\sigma\colon X\to X$, their algebra is defined as the groupoid C*-algebra induced by the topological groupoid~$\Gamma(X,T)$, where $T$ is the locally homeomorphic part of~$\sigma$.
Kajiwara and Watatani noted that Deaconu and Muhly used only the locally homeomorphic part of a branched covering.
This means that their algebras forget branch points.
Kajiwara and Watatani tried to create a C*-algebra which remembers branch points and introduced a C*-algebra associated with a complex dynamical system~\cite{KW05} and with a self-similar map~\cite{KW06} by the Cuntz--Pimsner construction~\cite{P}.

The Deaconu and Muhly's algebra for a branched covering $\sigma\colon X\to X$ has the natural Cartan subalgebra as a groupoid C*-algebra.
This is isomorphic to the C*-algebra~$C_0(X)$ of all continuous functions on a space $X$ vanishing at infinity.
Kajiwara and Watatani~\cite{KW17} showed that their algebra for a dynamical system on a space~$X$ also has a similar C*-subalgebra.
This is isomorphic to $C_0(X)$ and maximal abelian.
But it is unsolved whether this subalgebra is a Cartan subalgebra or not.

The notion of Cartan subalgebras were introduced by Kumjian \cite{Ku} and Renault \cite{R1980,R}.
According to Kumjian \cite{Ku}, the Cartan subalgebra was introduced under ``the belief that the structure of a C*-algebra is illuminated by an understanding of the manner in which abelian subalgebras embed in it.’’
By Renault’s theory, for a separable C*-algebra $A$ with a Cartan subalgebra $B$, there exists a topological groupoid $G$ with a unit space $G^{(0)}$ and with twist data $\Sigma$ and an isomorphism from the groupoid C*-algebra $C^*(G,\Sigma)$ onto $A$ carrying $C_0(G^{(0)})$ onto $B$.
Moreover, this twisted groupid $(G,\Sigma)$ is isomorphic to the Weyl twist $(G(B),\Sigma(B))$; and hence unique up to isomorphism.

Let~$R$ be a rational function with degree~$\geq 2$ and $X$ be its Julia set, its Fatou set, or the Riemann sphere.
Suppose that $X$ is not empty.
We can regard $R$ as a continuous map from~$X$ onto itself.
Let~$\O_R(X)$ denote the Kajiwara and Watatani's algebra associated with the dynamical system~$(X,R)$.
Kajiwara and Watatani showed in~\cite{KW17} that $C_0(X)$ is a maximal abelian subalgebra of $\O_R(X)$ in the case that $X$ is the Julia set, and left the problem of whether $C_0(X)$ is a Cartan subalgebra of~$\O_R(X)$ or not.

In this paper, we give the equivalent condition for~$C_0(X)$ to be a Cartan subalgebra of~$\O_R(X)$ and solve the problem.
To achieve the aim, we give an explicit universal covariant representation of Kajiwara--Watatani algebras on Hilbert space which does not require any quotients like the GNS representation.
This could be used for further analysis of Kajiwara–Watatani algebras.

\subsection{Main theorem}

A rational function is a function described as~$P/Q$ for some coprime polynomials~$P$ and~$Q$ and its degree is defined as the larger of the degrees of~$P$ and~$Q$.
A rational function is a holomorphic branched covering from the Riemann sphere to itself.

The Fatou set of a rational function~$R$ is the maximum one of open sets on which the iterations of~$R$ is normal.
The Julia set is the complement of Fatou set.
They are completely invariant, that is, their image and preimage under~$R$ coincides with themselves.

Let~$R$ be a rational function with $\deg\geq 2$ and $X$ be its Julia set, its Fatou set, or the Riemann sphere.
The Kajiwara and Watatani's algebra~$\O_R(X)$ is defined for the complex dynamical system~$(X,R)$.
Kajiwara and Watatani~\cite{KW17} showed that $C_0(X)$ is a maximal abelian subalgebra of~$\O_R(X)$ in the case that $X$ is the Julia set of~$R$.
In this paper, we develop their result and give a simple necessary and sufficient condition for~$C_0(X)$ to be a Cartan subalgebra of~$\O_R(X)$.
The aim of this paper is to prove the following theorem.

\begin{theorem}[Theorem~\ref{4.14}]
  Let~$R$ be a rational function with $\deg\geq 2$ and $X$ be its Julia set, its Fatou set, or the whole of the Riemann sphere.
  Suppose that $X$ is not empty.
  Then, the following are equivalent.
  
  (1)~$C_0(X)$ is a Cartan subalgebra of~$\O_R(X)$.
  
  (2)~$X$ has no branch point of~$R$.
\end{theorem}

\section{Preparations for Operator Algebras}
\subsection{Adjointable operators between right Hilbert modules}

In this sub-subsection, we review adjointable operators between right Hilbert modules. 
See~\cite{B} Part II, Section 7 for more detail.

Let~$A$ be a C*-algebra, $E$ and $F$ be right Hilbert $A$-modules, and $T\colon E\to F$ and $S\colon F\to E$ be linear maps. $S$ is called an \emph{adjointable operator of~$T$} if the equation $\inn{e}{Sf}_E=\inn{Te}{f}_F$ holds for any $e\in E$ and $f\in F$. $T$ is called \emph{adjointable} if there exists an adjoint of~$T$. An adjointable map is a bounded complex-linear operator and it is moreover $A$-linear. Let $\L(E,F)$ denote the set of all adjointable operators from~$E$ to~$F$. $\L(E,F)$ with operator norm is a Banach space.

Let $\Theta(E,F):=\set{\theta_{f,e}}{\text{$e\in E$ and $f\in F$}}\subseteq\L(E,F)$, where $\theta_{f,e}$ is defined as the operator which maps $x\in E$ to~$f\inn{e}{x}_E$. $\Theta(E,F)$ is closed under complex-scalar multiplication. Let $\K(E,F)$ denote the closed subspace of~$\L(E,F)$ generated by $\Theta(E,F)$. $\L(E):=\L(E,E)$ is a C*-algebra and $\K(E):=\K(E,E)$ is its closed ideal.

For the later argument, let $K(E)$ denote the algebraic additive subgroup of~$\L(E)$ generated by $\Theta(E):=\Theta(E,E)$.
Since $\Theta(E)$ is closed under complex-scalar multiplication, $K(E)$ is a linear subspace of~$\L(E)$.
Therefore, the closure of~$K(E)$ is~$\K(E)$.

\subsection{$X$-squared matrices}

Let $X$ be a set.
Let $(\e_x)_{x\in X}$ be the standard basis of the Hilbert space $\H_X:=\bigoplus_X\C$.
Let $M_X$ denote the C*-algebra of all bounded linear operators on~$\H_X$.
We call an element of~$M_X$ an \emph{$X$-squared matrix}.
For $x,y\in X$, \emph{the $(x,y)$-entry of $\alpha$} means the value~$\inn{\e_x}{\alpha\e_y}$.

\begin{definition}\label{2.1.2.1}
  Let $R$ be a binary relation on~$X$.

  \setlength{\hangindent}{2.5em}
  1) $R$ is \emph{the identity} if $xRy$ implies $x=y$ for any $x,y\in X$.
  
  \setlength{\hangindent}{2.5em}
  2) $R$ is \emph{one-to-one} if $xRy$ and $xRz$ imply $y=z$ and also $xRz$ and $yRz$ imply $x=y$ for any $x,y,z\in X$.
\end{definition}

\begin{definition}\label{2.1.2.2}
  For an $X$-squared matrix~$\alpha$, we define a binary relation~$\rel\alpha$ on~$X$ by the following: for any $x,y\in X$,
  \begin{equation*}
    x\rel\alpha y\defiff\text{the $(x,y)$-entry of~$\alpha$ does not vanish, i.e. $\inn{\e_x}{\alpha\e_y}\neq 0$}.
  \end{equation*}
\end{definition}

\begin{definition}\label{2.1.2.3}\ \par
  \hangindent=2.5em
  1) An $X$-squared matrix~$\alpha$ is \emph{diagonal} if the binary relation~$\rel\alpha$ is the identity.
  Let $D_X$ denote the set of all diagonal $X$-squared matrices.

  \hangindent=2.5em
  2) An $X$-squared matrix~$\alpha$ is \emph{quasi-monomial} if the binary relation~$\rel\alpha$ is one-to-one.
  Let $QM_X$ denote the set of all quasi-monomial $X$-squared matrices.
\end{definition}

\begin{remark}
From the definition above, it immediately follows that a quasi-monomial matrix is a matrix with at most one non-zero entry in each row and each column.
\end{remark}

\begin{proposition}\label{2.1.2.4}
  Let $\alpha$ and $\beta$ be $X$-squared matrices, $c$ be a complex number, and $x$ and $y$ be points in~$X$.

  \setlength{\hangindent}{2.5em}
  1) $x\rel{\alpha+\beta}y$ implies $x\rel\alpha y$ or $x\rel\beta y$.
  
  \setlength{\hangindent}{2.5em}
  2) $x\rel{\alpha c}y$ implies $x\rel\alpha y$.
  
  \setlength{\hangindent}{2.5em}
  3) $x\rel{\alpha\beta}z$ implies $x\rel\alpha\rel\beta z$.

  \setlength{\hangindent}{2.5em}
  4) $x\rel{\alpha^*}y$ implies $y\rel\alpha x$.
\end{proposition}
\begin{proof}\ \par
  \noindent
  1)
  Assume that $x\rel{\alpha+\beta}y$, i.e. $\inn{\e_x}{(\alpha+\beta)\e_y}\neq 0$.
  \begin{equation*}
    \inn{\e_x}{(\alpha+\beta)\e_y}
    =\inn{\e_x}{\alpha\e_y+\beta\e_y}
    =\inn{\e_x}{\alpha\e_y}+\inn{\e_x}{\beta\e_y}.
  \end{equation*}
  Therefore, by the assumption we have $\inn{\e_x}{\alpha\e_y}\neq 0$ or $\inn{\e_x}{\beta\e_y}\neq 0$, i.e. $x\rel\alpha y$ or $x\rel\beta y$.
  \medskip

  \noindent
  2)
  Assume that $x\rel{\alpha c}y$, i.e. $\inn{\e_x}{(\alpha c)\e_y}\neq 0$.
  \begin{equation*}
    \inn{\e_x}{(\alpha c)\e_y}
    =\inn{\e_x}{(\alpha\e_y)c}
    =\inn{\e_x}{\alpha\e_y}c.
  \end{equation*}
  Therefore, by the assumption we have $\inn{\e_x}{\alpha\e_y}c\neq 0$ and hence $\inn{\e_x}{\alpha\e_y}\neq 0$, i.e.~$x\rel\alpha y$.
  \medskip

  \noindent
  3)
  Assume that $x\rel{\alpha\beta}z$, i.e. $\inn{\e_x}{\alpha\beta\e_z}\neq 0$.
  \begin{equation*}
    \inn{\e_x}{\alpha\beta\e_z}
    =\inn{\e_x}{\alpha\sum_{y\in X}\e_y\inn{\e_y}{\beta\e_z}}
    =\sum_{y\in X}\inn{\e_x}{\alpha\e_y\inn{\e_y}{\beta\e_z}}
    =\sum_{y\in X}\inn{\e_x}{\alpha\e_y}\inn{\e_y}{\beta\e_z}.
  \end{equation*}
  Therefore, by the assumption there exists $y\in X$ such that $\inn{\e_x}{\alpha\e_y}\inn{\e_y}{\beta\e_z}\neq 0$, which implies both of $\inn{\e_x}{\alpha\e_y}\neq 0$ and $\inn{\e_y}{\beta\e_z}\neq 0$, i.e. both of $x\rel\alpha y$ and $y\rel\beta z$.
  So we have $x\rel\alpha\rel\beta z$.
  \medskip

  \noindent
  4)
  Assume that $x\rel{\alpha^*}y$, i.e. $\inn{\e_x}{\alpha^*\e_y}\neq 0$.
  Then, we have $\inn{\e_y}{\alpha\e_x}=\overline{\inn{\e_x}{\alpha^*\e_y}}\neq 0$, which implies $y\rel\alpha x$.
\end{proof}

\begin{fact}[\cite{B}, II.6.10.3 Corollary]\label{2.1.2.5}
  Let $A$ be a C*-algebra, $B$ a C*-algebra, and $\theta\colon A\to B$ an idempotent positive $B$-linear map.
  Then $\theta$ is a conditional expectation.
\end{fact}

\begin{proposition}\label{2.1.2.6}
  Let $\delta\colon M_X\to D_X$ be the complex-linear operator defined by
  \begin{equation*}
    \delta(\alpha)\e_x:=\e_x\inn{\e_x}{\alpha\e_x}\quad\text{for each $\alpha\in M_X$ and $x\in X$.}
  \end{equation*}
  Then the following five statements hold.
  \medskip
  
  \hangindent=2.5em
  1) $\delta(\alpha)=\alpha$ holds for any $\alpha\in D_X$.
  
  \hangindent=2.5em
  2) $\delta$ is idempotent, i.e. $\delta\circ\delta=\delta$.
  
  \hangindent=2.5em
  3) $\delta$ is $D_X$-linear.
  
  \hangindent=2.5em
  4) $\delta$ is positive as an operator.
  That is, for any positive element~$\alpha$ in~$M_X$, $\delta(\alpha)$ is a positive element of~$D_X$. Therefore, $\delta$ is also bounded.
  
  \hangindent=2.5em
  5) $\delta$ is faithful as a positive operator.
  That is, an positive element~$\alpha$ of~$M_X$ satisfying that $\delta(\alpha)=0$ is only the zero element.
  \medskip

  \noindent
  Therefore, $\delta$ is a conditional expectation.
\end{proposition}
\begin{proof}\ \par
  \noindent
  1) For any $x\in X$, we have $\delta(\alpha)\e_x=\e_x\inn{\e_x}{\alpha\e_x}=\sum_{y\in X}\e_y\inn{\e_y}{\alpha\e_x}=\alpha\e_x$ and hence $\delta(\alpha)=\alpha$.
  
  \noindent
  2)
  For any $\alpha\in M_X$, letting $\beta:=\delta(\alpha)$, we have $(\delta\circ\delta)(\alpha)=\delta(\beta)=\beta=\delta(\alpha)$.
  Therefore the operator $\delta$ is idempotent.

  \noindent
  3) 
  Let $\alpha,\beta\in M_X$.
  The following holds for any $x\in X$; hence $\delta(\alpha+\beta)=\delta(\alpha)+\delta(\beta)$.
  \begin{gather*}
    \delta(\alpha+\beta)\e_x
    =\e_x\inn{\e_x}{(\alpha+\beta)\e_x}
    =\e_x\inn{\e_x}{\alpha\e_x}+\e_x\inn{\e_x}{\beta\e_x}\\
    =\delta(\alpha)\e_x+\delta(\beta)\e_x
    =(\delta(\alpha)+\delta(\beta))\e_x.
  \end{gather*}
  
  Let $\alpha\in M_X$ and $\beta\in D_X$.
  Applying the result of 1), we have $\delta(\beta)=\beta$.
  The following hold for any $x\in X$:
  \begin{gather*}
    \delta(\alpha\beta)\e_x
    =\e_x\inn{\e_x}{\alpha\beta\e_x}
    =\e_x\inn{\e_x}{\alpha\sum_{y\in X}\e_y\inn{\e_y}{\beta\e_x}}
    =\sum_{y\in X}\e_x\inn{\e_x}{\alpha\e_y\inn{\e_y}{\beta\e_x}}
    \\=\sum_{y\in X}\e_x\inn{\e_x}{\alpha\e_y}\inn{\e_y}{\beta\e_x}
    =\e_x\inn{\e_x}{\alpha\e_x}\inn{\e_x}{\beta\e_x}
    =\delta(\alpha)\e_x\inn{\e_x}{\beta\e_x}
    =\delta(\alpha)\delta(\beta)\e_x
  \end{gather*}
  and
  \begin{gather*}
    \delta(\beta\alpha)\e_x
    =\e_x\inn{\e_x}{\beta\alpha\e_x}
    =\e_x\inn{\e_x}{\beta\sum_{y\in X}\e_y\inn{\e_y}{\alpha\e_x}}
    =\sum_{y\in X}\e_x\inn{\e_x}{\beta\e_y\inn{\e_y}{\alpha\e_x}}
    \\=\sum_{y\in X}\e_x\inn{\e_x}{\beta\e_y}\inn{\e_y}{\alpha\e_x}
    =\e_x\inn{\e_x}{\beta\e_x}\inn{\e_x}{\alpha\e_x}
    =\delta(\beta)\e_x\inn{\e_x}{\alpha\e_x}
    =\delta(\beta)\delta(\alpha)\e_x.
  \end{gather*}
  Therefore $\delta(\alpha\beta)=\delta(\alpha)\delta(\beta)=\delta(\alpha)\beta$ and $\delta(\beta\alpha)=\delta(\beta)\delta(\alpha)=\beta\delta(\alpha)$ hold.

  \noindent
  4) 
  Let $\alpha$ be a positive element of~$M_X$, i.e. a positive operator on~$\H_X$.
  Then the inequality $\inn{\e_x}{\alpha\e_x}\geq 0$ holds for any $x\in X$.
  So $\delta(\alpha)$ is a diagonal matrix whose all of diagonal entries are non-negative real numbers and hence is a positive element of~$D_X$.
  Since this holds for any positive element $\alpha$ of~$M_X$, $\delta$ is a positive operator.

  \noindent
  5) 
  Let $\alpha$ be a positive element of~$M_X$ with $\delta(\alpha)=0$.
  The following holds for any $x\in X$ and hence all of diagonal entries of $\alpha$ are zero.
  \begin{equation*}
    \inn{\e_x}{\alpha\e_x}
    =\inn{\e_x}{\e_x}\inn{\e_x}{\alpha\e_x}
    =\inn{\e_x}{\e_x\inn{\e_x}{\alpha\e_x}}
    =\inn{\e_x}{\delta(\alpha)\e_x}
    =\inn{\e_x}{0\e_x}
    =0.
  \end{equation*}
  Let $x$ and $y$ be two distinct points in $X$ and let $c:=\inn{\e_x}{\alpha\e_y}$.
  Since $\alpha$ is self-adjoint, we have ${\bar c}=\inn{\e_y}{\alpha\e_x}$.
  Let $\v:=c\e_x-|c|\e_y$.
  \begin{align*}
    \inn{\v}{\alpha\v}
    &={\bar c}\inn{\e_x}{\alpha\e_x}c
      +{\bar c}\inn{\e_x}{\alpha\e_y}(-|c|)
      +\overline{-|c|}\inn{\e_y}{\alpha\e_x}c
      +\overline{-|c|}\inn{\e_y}{\alpha\e_y}(-|c|)\\
    &={\bar c}0c
      +{\bar c}c(-|c|)
      +\overline{-|c|}{\bar c}c
      +\overline{-|c|}0(-|c|)\\
    &=-2|c|^3.
  \end{align*}
  Since $\alpha$ is a positive element of~$M_X$, i.e. a positive operator on~$\H_X$, we obtain that $\inn{\v}{\alpha\v}\geq 0$.
  Hence we have $\inn{\e_x}{\alpha\e_y}=c=0$.
  Since this equation holds for any two distinct points $x$ and $y$ in~$X$, the matrix $\alpha$ is diagonal.

  By the above results, $\alpha$ is a diagonal matrix whose diagonal entries are zero; hence it is the zero matrix.
\end{proof}

\subsection{Cuntz--Pimsner algebras}

In this sub-subsection, we review Cuntz--Pimsner algebras, which were introduced in \cite{P}.
See \cite{P} or \cite{K} for more detail.

A \emph{C*-correspondence} over a C*-algebra~$A$ is a right Hilbert $A$-module~$E$ endowed with the left $A$-scalar multiplication induced by a $*$-homomorphism from~$A$ to~$\L(E)$.

Let $A$ and $B$ be C*-algebras, $E$ be a C*-correspondence over $A$, and $\rho_A\colon A\to B$ and $\rho_E\colon E\to B$ be maps.
The pair $\rho=(\rho_A,\rho_E)$ is called a \emph{representation of $E$} if $\rho_A$ is a $*$-homomorphism, $\rho_E$ is a complex-linear map, and moreover the following two conditions hold.
\medskip

\hangindent=2.5em
(1) $\rho$ preserves inner product.
That is, $\rho_A(\inn{f}{g})=\rho_E(f)^*\rho_E(g)$ holds for any $f,g\in E$.

\hangindent=2.5em
(2) $\rho$ preserves left $A$-scalar multiplication.
That is, $\rho_E(af)=\rho_A(a)\rho_E(f)$ holds for any $a\in A$ and $f\in E$.

\begin{fact}[\cite{K}, after Definition 2.1]
If $\rho$ is a representation of~$E$, then the next holds automatically:
\smallskip

\hangindent=2.5em
(3) $\rho$ preserves right $A$-scalar multiplication.
That is, $\rho_E(fa)=\rho_E(f)\rho_A(a)$ holds for any $a\in A$ and $f\in E$.
\end{fact}
\begin{proof}
Let $b=\rho_E(fa)-\rho_E(f)\rho_A(a)$.
Since $\rho$ preserves the inner product and $\rho_A$ is a *-homomorphism, we have
\begin{align*}
b^*b
&=\bigl(\rho_E(fa)-\rho_E(f)\rho_A(a)\bigr)^*\bigl(\rho_E(fa)-\rho_E(f)\rho_A(a)\bigr)\\
&=\rho_E(fa)^*\rho_E(fa)-\rho_E(fa)^*\rho_E(f)\rho_A(a)\\
&\qquad-\rho_A(a)^*\rho_E(f)^*\rho_E(fa)+\rho_A(a)^*\rho_E(f)^*\rho_E(f)\rho_A(a)\\
&=\rho_A(\inn{fa}{fa})-\rho_A(\inn{fa}{f})\rho_A(a)\\
&\qquad-\rho_A(a^*)\rho_A(\inn{f}{fa})+\rho_A(a^*)\rho_A(\inn{f}{f})\rho_A(a)\\
&=\rho_A(\inn{fa}{fa}-\inn{fa}{f}a-a^*\inn{f}{fa}+a^*\inn{f}{f}a)\\
&=0.
\end{align*}
So we have $\norm{b}^2=\norm{b^*b}=0$ and hence $b=0$.
\end{proof}

For a representation $\rho=(\rho_A,\rho_E)$, let $C^*(\rho)$ denote the C*-subalgebra of~$B$ generated by the set $\rho_A(A)\cup\rho_E(E)$.

A representation $\rho$ is called \emph{injective} if $\rho_A$ is injective.
In this case, $\rho_E$ is also injective automatically.

Let $A$ be a C*-algebra, $E$ be a C*-correspondence over $A$, and $\rho=(\rho_A,\rho_E)$ be a representation of~$E$.
We define two $*$-homomorphisms $\varphi$ and $\psi_\rho$ by the following.
\begin{alignat*}{3}
  &\varphi\colon A\to\L(E),
  &\quad&\varphi(a)(f):=af
  &\quad&\text{for each $a\in A$, $f\in E$}.
  \\
  &\psi_\rho\colon\K(E)\to C^*(\rho),
  &\quad&\psi_\rho(\theta_{f,g}):=\rho_E(f)\rho_E(g)^*
  &\quad&\text{for each $f,g\in E$}.
\end{alignat*}
Suppose that $\varphi$ is injective.
Let $\I_E:=\set{a\in A}{\varphi(a)\in\K(E)}$.
This is a closed ideal of~$A$ and is called a \emph{covariant ideal} of $E$.
We call $\rho$ \emph{covariant} if $\rho_A(a)=\psi_\rho(\varphi(a))$ holds for any $a\in\I_E$.
There exist covariant representations whenever $\varphi$ is injective, and also exists the universal one of them.
\emph{The Cuntz--Pimsner algebra associated with~$E$} is the C*-algebra induced by the universal covariant representation of~$E$ and denoted by~$\O_E$.

\begin{remark}
The condition for injectivity of $\varphi$ is not necessary in the definition modified by Katsura in \cite{K}.
\end{remark}

\begin{fact}[\cite{K}, Lemma 2.4]\label{2.1.3.1}
  The equation $\psi_\rho(\theta)\rho_E(f)=\rho_E(\theta f)$ holds for any $\theta\in\K(E)$ and $f\in E$.
\end{fact}
\begin{proof}
  The equation holds for any $\theta\in\Theta(E)$ since we have the following for any $f,g,h\in E$:
  \begin{equation*}
    \psi_\rho(\theta_{f,g})\rho_E(h)
    =\rho_E(f)\rho_E(g)^*\rho_E(h)
    =\rho_E(f\inn{g}{h})
    =\rho_E(\theta_{f,g}h).
  \end{equation*}
  By the linearity, it also holds for any $\theta\in K(E)$.
  Since two linear operators $\theta\mapsto\psi_\rho(\theta)\rho_E(f)$ and $\theta\mapsto\rho_E(\theta f)$ are bounded, the equation holds for any $\theta\in\K(E)$.
\end{proof}

\begin{proposition}\label{2.1.3.2}
  The equation $\psi_\rho(\varphi(a))\rho_E(f)=\rho_A(a)\rho_E(f)$ holds for any $a\in\I_E$ and $f\in E$.
\end{proposition}
\begin{proof}
  Using Fact \ref{2.1.3.1}, we have
  \begin{equation*}
    \psi_\rho(\varphi(a))\rho_E(f)
    =\rho_E(\varphi(a)f)
    =\rho_E(af)
    =\rho_A(a)\rho_E(f).
    \qedhere
  \end{equation*}
\end{proof}

\begin{definition}[\cite{K}, Definition 5.6]\label{2.1.3.3}
  A representation $\rho=(\rho_A,\rho_E)$ of a C*-correspondence~$E$ over a C*-algebra~$A$ is said to \emph{admit a gauge action} if for each complex number~$c$ with $\abs{c}=1$, there exists a $*$-homomorphism $\Phi_c\colon C^*(\rho)\to C^*(\rho)$ such that $\Phi_c(\rho_A(a))=\rho_A(a)$ and $\Phi_c(\rho_E(f))=c\rho_E(f)$ for all $a\in A$ and $f\in E$.
\end{definition}

\begin{fact}[\cite{K}, Theorem 6.4]\label{2.1.3.4}
  For a covariant representation $\rho$ of a C*-correspondence $E$, the $*$-homomorphism $\O_E\to C^*(\rho)$ is an isomorphism if and only if $\rho$ is injective and admits a gauge action.
\end{fact}

\begin{proposition}\label{2.1.3.5}
  The equation $\lim_\lambda(f\cdot a_\lambda)=f$ holds for any approximate unit $\{a_\lambda\}$ of~$A$ and $f\in E$.
\end{proposition}
\begin{proof}
  Since $\inn{f}{f}$ is a positive element of~$A$, we can let $\abs{f}:=\sqrt{\inn{f}{f}}$.
  \begin{gather*}
    \norm{fa_\lambda-f}_E^2
    =\norm{f(a_\lambda-1)}_E^2
    =\inn{f(a_\lambda-1)}{f(a_\lambda-1)}\\
    =(a_\lambda-1)^*\inn{f}{f}(a_\lambda-1)
    =(a_\lambda-1)^*\abs{f}^2(a_\lambda-1)\\
    =\NORM\Big{\abs{f}(a_\lambda-1)}_A^2
    =\NORM\Big{\abs{f}a_\lambda-\abs{f}}_A^2
  \end{gather*}
  Therefore, we have $\lim_\lambda\norm{fa_\lambda-f}_E=0$, which implies that $\lim_\lambda(f\cdot a_\lambda)=f$.
\end{proof}

\subsection{$Z$-squared $B$-valued matrices}

Let $Z$ be a set, $B$ be a C*-algebra and $B_1$ be its unitization.
Let $B_1^{\oplus Z}$ denote the algebraic direct sum of $\abs{Z}$-many copies of $B_1$ as a right $B_1$-module and $\H_Z(B)$ be the completion of the right pre-Hilbert $B_1$-module defined as $B_1^{\oplus Z}$ endowed with the $B_1$-valued inner product defined by $\inn{\boldsymbol a}{\boldsymbol b}=\sum_{n\in Z}a_n^*b_n$ for any ${\boldsymbol a}=(a_n)_{n\in Z}$ and ${\boldsymbol b}=(b_n)_{n\in Z}$.
Then, $\H_Z(B)$ becomes a right Hilbert $B_1$-module (see \cite{B}, II.7.1.5 Proposition and II.7.1.7 Examples (iv)).
We note that $\H_Z(\C)=\H_Z$.
For the argument in this sub-subsection, let $\e_m:=(\delta_{m,n})_{n\in Z}$ for each $m\in Z$.

Let $M_Z(B)$ denote the C*-algebra~$\L(\H_Z(B))$ of all adjointable operators from $\H_Z(B)$ to itself.
We call an element of~$M_Z(B)$ a \emph{$Z$-squared $B$-valued matrix}.
We note that $M_Z(\C)=M_Z$.

For $m,n\in Z$, \emph{the $(m,n)$-entry} of a $Z$-squared $B$-valued matrix~$\alpha$ is the value $\inn{\e_m}{\alpha\e_n}$.
For the argument in sub-subsection part, let $\epsilon_{m,n}$ be the matrix whose $(m,n)$-entry is~$1$ and the others are~$0$.

\begin{remark}
  An entry of a $Z$-squared $B$-valued matrix may not be in~$B$ but in~$B_1$.
\end{remark}

\begin{proposition}\label{2.1.4.1}
  A $Z$-squared $B$-valued matrix with no non-zero entries is zero.
\end{proposition}
\begin{proof}
  We note that the closed $B_1$-linear span of the~$\e_n$ coincides with~$\H_Z(B)$.
  A $Z$-squared $B$-valued matrix with no non-zero entries maps any finite linear combination of the~$\e_n$ to zero and hence any element of~$\H_Z(B)$ to zero.
  Therefore, such a matrix is zero.
\end{proof}

\begin{proposition}\label{2.1.4.2}
  $M_Z(B)$ includes~$M_Z$.
  That is, for any $Z$-squared matrix~$\alpha$, there exists a $Z$-squared $B$-valued matrix~$\beta$ such that $\inn{\e_m}{\alpha\e_n}=\inn{\e_m}{\beta\e_n}$ holds for all $m,n\in Z$.
\end{proposition}

\begin{proof}
  We note that the~$\epsilon_{m,n}$ belong to both $M_Z$ and $M_Z(B)$.
  Since $\H_Z$ is included in~$\H_Z(B)$, the norm of~$M_Z$ is less than or equal to the norm of~$M_Z(B)$.
  Hence, the identity map on the algebraic complex-linear span of the~$\epsilon_{m,n}$ has a continuous extension from the closure in the norm of~$M_Z(B)$ to the closure in the norm of~$M_B$.
  The two closures are C*-subalgebras respectively and the extension is a $*$-homomorphism between them.
  In particular, the codomain one coincides with the whole of~$M_B$.
  By Proposition~\ref{2.1.4.1}, the extension is injective and hence admits the inverse.
  Under the inverse, $M_Z$ embeds into~$M_Z(B)$.
\end{proof}

\subsection{The reduced group C*-algebra $C^*(\Z)$}

For each $n\in\Z$, we define a $\Z$-squared matrix $\upsilon_n$ by $\upsilon_n\e_k:=\e_{n+k}$ for all $k\in\Z$.
The reduced group C*-algebra~$C^*(\Z)$ is the C*-subalgebra of~$M_\Z$ generated by $\set{\upsilon_n}{n\in\Z}$.

\begin{proposition}\label{2.1.5.2}
Let $c$ be a complex number with absolute value~$1$.
Then, there exists a unique *-homomorphism $\Phi_c\colon C^*(\Z)\to C^*(\Z)$ satisfying the equation $\Phi_c(\upsilon_n)=\upsilon_n c^n$ for all $n\in\Z$.
\end{proposition}

\begin{proof}
Uniqueness follows immediately from that $C^*(\Z)$ is generated by the $\upsilon_n$.
We show existence of $\Phi_c$.
Define a $\Z$-squared matrix~$\omega_c$ by $\omega_c\e_k:=\e_k c^k$ for all $k\in\Z$.
This is unitary.
Let $\Phi_c$ be the map from $C^*(\Z)$ to itself defined by $\Phi_c(\alpha):=\omega_c\alpha\omega_c^*$ for each $\alpha\in C^*(\Z)$.
This is a *-homomorphism.
The following holds for any $n,k\in\Z$:
\begin{equation*}
  \Phi_c(\upsilon_n)\e_k
  =\omega_c\upsilon_n\omega_c^*\e_k
  =\omega_c\upsilon_n\e_k{\bar c}\,^k
  =\omega_c\e_{n+k}{\bar c}\,^k
  =\e_{n+k}c^{n+k}{\bar c}\,^k
  =\upsilon_n\e_k c^n.
  \qedhere
\end{equation*}
\end{proof}

We consider the conditional expectation $\delta_\Z\colon M_\Z\to D_\Z$ defined in Proposition~\ref{2.1.2.6}.
This maps $\upsilon_0$ to~$1_{M_\Z}$ and $\upsilon_n$ ($n\neq  0$) to zero and hence the image of $C^*(\Z)$ is $\C 1_{M_\Z}\cong\C$.
For each $n\in\Z$, we define a bounded linear functional $\delta_n\colon C^*(\Z)\to\C$ by $\delta_n(\alpha):=\delta_\Z(\alpha\upsilon_{-n})$ for all $\alpha\in C^*(\Z)$.

\begin{proposition}\label{2.1.5.1}
  Let $B$ be a C*-algebra and $\xi$ be an element of the tensor product $B\otimes C^*(\Z)$.
  Assume that $(\id\otimes\delta_n)(\xi)=0$ for all $n\in\Z$.
  Then, $\xi=0$.
\end{proposition}
\begin{proof}
  By Proposition~\ref{2.1.4.2}, we have $C^*(\Z)\subseteq M_\Z\subseteq M_\Z(B)$.
  We embed $B$ into $M_\Z(B)$ under the injective $*$-homomorphism $b\mapsto 1_{M_\Z(B)}\cdot b$.
  Under these embeddings, $B$ and $C^*(\Z)$ commute in~$M_\Z(B)$.
  So there exists a unique natural $*$-homomorphism $B\otimes C^*(\Z)\to M_\Z(B)$.
  By the nuclearity of $C^*(\Z)$, $B\otimes C^*(\Z)$ is isomorphic to its image under the $*$-homomorphism.
  
  The assumption of this proposition implies that every entry of the $\Z$-squared $B$-valued matrix~$\xi$ vanishes.
  Therefore, by Proposition~\ref{2.1.4.1}, we obtain that $\xi=0$.
\end{proof}

\section{Review of Complex Dynamical Systems}
\subsection{Branch points of holomorphic functions}

In this sub-subsection, we review branch points of holomorphic functions.
See \cite{F} for more detail.

\emph{The branch index} of a non-constant holomorphic function~$R$ at a point~$x$ in its domain, which is denoted by~$e_R(x)$ in this paper, is the minimum of the degrees of the non-constant terms with non-zero coefficients in the power series representation of~$R$ around~$x$.
The notion is derived from the next well-known fact.

\begin{fact}\label{2.2.1.1}
  Let $R$ be a non-constant holomorphic function with open domain~$D$ in the complex plain and $x$ be a point in~$D$.
  Suppose $y:=R(x)$ and $k:=e_R(x)$.
  Then there exists charts $\varphi\colon U\to V$ on~$D$ and $\psi\colon U'\to V'$ on the complex plane with the following properties.
  \begin{center}
    $x\in U$,\quad
    $\varphi(x)=0\in V$,\quad
    $y\in U'$,\quad
    $\psi(y)=0\in V'$,\quad
    $R(U)\subseteq U'$,\quad
    and $\psi\circ R\circ\varphi^{-1}(z)=z^k$ for all $z\in V$.
  \end{center}
\end{fact}
\begin{proof}
See \cite{F}, 2.1. Theorem.
\end{proof}

\begin{corollary}\label{2.2.1.2}
  Let $R$ be a non-constant holomorphic function with an open domain in the complex plain and $x$ be a point in the domain.
  Then the following are equivalent.
  
  \hangindent=2.5em
  (1) $x$ is a \emph{branch point of~$R$}, i.e. the restriction of~$R$ to any open neighborhood of~$x$ is not injective.
  
  \hangindent=2.5em
  (2) $x$ is a \emph{critical point of $R$}, i.e. the derivative of~$R$ vanishes at~$x$.
  
  \hangindent=2.5em
  (3) The branch index of~$R$ at~$x$ is greater than~$1$. 
\end{corollary}

\begin{fact}[\cite{F}, 2.4. Corollary]\label{2.2.1.3}
  Let $X$ and $Y$ be Riemann surfaces and let $R\colon X\to Y$ be a non-constant holomorphic mapping.
  Then $R$ is open, i.e., the image of every open set under~$R$ is open.
\end{fact}

\subsection{Branch points and fixed points of rational functions}

A rational function is a function which can be written as a quotient of polynomials.
This is a holomorphic function from the Riemann sphere to itself.
The degree of a rational function~$R$ is the larger of the degrees of $P$ and $Q$, where $P$ and $Q$ are coprime polynomials with~$R=P/Q$.

The next proposition follows from the compactness of the Riemann sphere and the identity theorem.

\begin{proposition}\label{2.2.2.1}\ \par
  (1) A rational function with infinitely many critical points is a constant function.
  
  (2) A rational function with infinitely many fixed points is an identity function.
\end{proposition}

In particular, a rational function with~${\deg}\geq 2$ has only finite many branch points and  fixed points.

\subsection{Fatou and Julia sets}

In this sub-subsection, we review the Fatou and Julia set of a rational function.
See \cite{CG}, Chapter III for more detail.

Let $R$ be a rational function.
Open sets~$U$ on which the family of all iterations of~$R$ is a normal family are closed under union, and hence there exists the largest one of such open sets.
The largest one is called \emph{the Fatou set of~$R$} and its complement in the Riemann sphere is called \emph{the Julia set of $R$}.

\begin{fact}[\cite{CG}, Section III Theorem 1.2, 1.3, 1.8]\label{2.2.3.1}
  The Julia set~$J$ of a rational function~$R$ with~$\deg\geq 2$ is
  \begin{itemize}
    \item non-empty,
    \item completely invariant under~$R$ (i.e. $R^{-1}(J)=J=R(J)$), and
    \item perfect (i.e. $J$ has no isolated points).
  \end{itemize}
\end{fact}

\begin{example}[\cite{CG}, p.55]\label{2.2.3.2}\ \par
  (1) The Julia set of $R(z)=z^2$ is the unit circle~$|z|=1$.
  
  (2) The Julia set of $R(z)=z^2-2$ is the closed interval~$[-2,2]$.
\end{example}

\section{Review of Kajiwara--Watatani Algebras}

In this subsection, we briefly review the definition of C*-algebras associated with complex dynamical systems introduced in~\cite{KW05}.

Let
  $R$ be a rational function with~$\deg\geq 2$ and
  $X$ be its Julia set, its Fatou set, or the whole of the Riemann sphere.
Suppose that $X$ is not empty.
We note that $X$ is locally compact, completely invariant under $R$, and perfect.

Hereafter, we regard $R$ as a map from~$X$ to itself and consider the dynamical system $(X,R)$.

Let $A:=C_0(X)$ as a C*-algebra and let $E$ be the completion of the pre-Hilbert $A$-$A$-bimodule defined as the complex-linear space $C_c(X)$ endowed with the left and right $A$-scalar multiplication and $A$-valued inner product defined by the following.
\begin{equation*}
  \begin{array}{l}
    (a\cdot f)(x):=a(x)f(x)\\
    (f\cdot a)(x):=f(x)a(R(x))
  \end{array}
  \quad\text{for each $a\in A$, $f\in C_c(X)$ and $x\in X$}.
\end{equation*}
\begin{equation*}
  \inn{f}{g}(y):=\stacksum{x\in X}{R(x)=y}\hspace{-0.5em}e_R(x)\overline{f(x)}g(x)\quad\text{for each $f,g\in C_c(X)$ and $y\in X$}.
\end{equation*}
Then, $E$ is a C*-correspondence over~$A$, which is called \emph{the Kajiwara--Watatani C*-correspondence}.

The left action $A\to\L(E)$, $a\mapsto (f\mapsto a\cdot f)$ is injective.
Therefore, we can define the Cuntz--Pimsner algebra associated with the C*-correspondence~$E$ over~$A$.
The C*-algebra is called \emph{the Kajiwara--Watatani algebra} associated with the complex dynamical system~$(X,R)$ and denoted by~$\O_R(X)$.

\begin{remark}
  For any $f\in C_c(X)$, the inequality $\norm{f}_{C_0(X)}\leq\norm{f}_{E}\leq\sqrt{\deg R}\cdot\norm{f}_{C_0(X)}$ holds and hence the two norms on~$C_c(X)$ are equivalent.
  Therefore, $E$ coincides with~$C_0(X)$ as a complex-linear space.
\end{remark}

\section{Certain Representations of Kajiwara--Watatani C*-correspondences}
\subsection{Definition and fundamental properties of representation $\rho$}

We define a equivalence relation~$\sim$ on~$X$ by the following: for each $x,y\in X$,
\begin{equation*}
  x\sim y\defiff\text{$R^{\circ m}(x)=R^{\circ n}(y)$ holds for some $m,n\in\N\sqcup\{0\}$.}
\end{equation*}
Here, $R^{\circ n}$ denotes the $n$-th iterate of~$R$, that is, $R^{\circ 0}=\id$, $R^{\circ 1}=R$, $R^{\circ 2}=R\circ R$, and so on.
In this paper, we call the equivalence class of a point~$x$ under the relation~$\sim$ \emph{the orbit of~$x$ under~$R$}.
Since $R$ is a rational function, the orbit of a point under~$R$ is countable and completely invariant under~$R$.
\emph{The orbit of a subset~$S$ under~$R$} is the union of orbits of elements of~$S$, which is the minimum of completely invariant subsets including~$S$.

Consider the orbit under~$R$ of the set of all the fixed points and branch points of the iterations of~$R\colon X\to X$.
This is countable since each iteration of~$R$ has only finite many branch points and fixed points (Proposition \ref{2.2.2.1}).
Let $\X$ be the complement of the orbit.
Here we recall that $X$ has no isolated points.
So $\X$ is comeager and hence dense in $X$ by the Baire category theorem, and also completely invariant under~$R$.

We define maps $\rho_A\colon A\to M_\X$ and $\rho_E\colon E\to M_\X$ by the following:
\begin{alignat*}{3}
  &\rho_A(a)\e_x
  &&:=\e_xa(x)
  &\quad&\text{for each $a\in A=C_0(X)$ and $x\in\X$.}\\
  &\rho_E(f)\e_y
  &&:=\stacksum{x\in\X}{R(x)=y}\hspace{-0.5em}\e_xf(x)
  &\quad&\text{for each $f\in E=C_0(X)$ and $y\in\X$.}
\end{alignat*}

\begin{proposition}
$\rho=(\rho_A,\rho_E)$ is a representation of the C*-correspondence~$E$ over~$A$.
\end{proposition}
\begin{proof}
It is easy to verify that $\rho_A$ is a *-homomorphism and that $\rho_E$ is a complex-linear map.

We show that $\rho$ preserves the inner product.
Take any $f,g\in E$ and $y,y’\in\X$.
We want to show the equation $\inn{\e_y}{\rho_E(f)^*\rho_E(g)\e_{y’}}=\inn{\e_y}{\rho_A(\inn{f}{g})\e_{y’}}$.
We have 
\begin{align*}
  (\text{LHS})
  &=\inn{\rho_E(f)\e_y}{\rho_E(g)\e_{y’}}\\
  &=\INN\bigg{\sum_{\substack{x\in\X\\R(x)=y}}\e_xf(x)}{\sum_{\substack{z\in\X\\R(z)=y’}}\e_zg(z)}\\
  &=\sum_{\substack{x\in\X\\R(x)=y}}\sum_{\substack{z\in\X\\R(z)=y’}}\overline{f(x)}\inn{\e_x}{\e_z}g(z)\\
  &=\begin{cases}
    \displaystyle\sum_{\substack{x\in\X\\R(x)=y}}\overline{f(x)}g(x)&\text{if $y=y’$,}\\
    0&\text{otherwise.}
  \end{cases}
\end{align*}
And also we have
\begin{align*}
  (\text{RHS})
  &=\inn{\e_y}{\e_{y’}}\inn{f}{g}(y’)\\
  &=\inn{\e_y}{\e_{y’}}\inn{f}{g}(y)\\
  &=\inn{\e_y}{\e_{y’}}\sum_{\substack{x\in X\\R(x)=y}}e_R(x)\overline{f(x)}g(x).\\
  \intertext{Since $y\in\X$ and $\X$ is completely invariant under $R$, we have}
  &=\inn{\e_y}{\e_{y’}}\sum_{\substack{x\in\X\\R(x)=y}}e_R(x)\overline{f(x)}g(x).\\
  \intertext{By the definition of $\X$, branch points of $R$ are not in $\X$ and hence the branch index $e_R$ takes value $1$ constantly on $\X$. So we have}
  &=\inn{\e_y}{\e_{y’}}\sum_{\substack{x\in\X\\R(x)=y}}\overline{f(x)}g(x).
\end{align*}
So $(\text{LHS})=(\text{RHS})$ holds.
Since it holds for any $y,y’\in\X$, we have $\rho_E(f)^*\rho_E(g)=\rho_A(\inn{f}{g})$.
Therefore, $\rho$ preserves the inner product.

We show that $\rho$ preserves left $A$-scalar multiplication.
Take any $f,g\in E$.
For any $y\in\X$, we have 
\begin{gather*}
\rho_E(af)\e_y
=\sum_{\substack{x\in\X\\R(x)=y}}\e_xa(x)f(x)
=\sum_{\substack{x\in\X\\R(x)=y}}\rho_A(a)\e_xf(x)\\
=\rho_A(a)\sum_{\substack{x\in\X\\R(x)=y}}\e_xf(x)
=\rho_A(a)\rho_E(f)\e_y.
\end{gather*}
So we obtain that $\rho_E(af)=\rho_A(a)\rho_E(f)$.
Therefore, $\rho$ preserves left $A$-scalar multiplication.

By the above discussion, $\rho$ is a representation of $E$.
\end{proof}

\begin{proposition}\label{3.1.1}
  The representation~$\rho$ is injective.
\end{proposition}
\begin{proof}
  We only show the injectivity of~$\rho_A$.
  Take any $a\in\ker\rho_A$.
  For any point $x\in\X$, we have 
  \begin{equation*}
    a(x)
    =\inn{\e_x}{\rho_A(a)\e_x}
    =\inn{\e_x}{0\e_x}
    =0.
  \end{equation*}
  Since $\X$ is dense in~$X$, we have $a=0$.
  Therefore, $\rho_A$ is injective.
\end{proof}

\begin{lemma}\label{3.1.2}
  Let $\alpha$ be an $\X$-squared matrix.
  Assume that $\alpha\rho_E(f)=0$ holds for all $f\in E$.
  Then, $\alpha=0$ holds.
\end{lemma}
\begin{proof}
  Take any $x\in\X$.
  Since the set $R^{-1}(R(x))$ is finite, there exists a continuous function~$f$ on~$X$ with the following conditions:
  \begin{equation*}
    f(x)=1,\qquad
    f(y)=0\quad\text{for each $y\in\inv{R}(R(x))\setminus\{x\}$}.
  \end{equation*}
  We have
  \begin{equation*}
    \alpha\e_x
    =\alpha\hspace{-1em}\stacksum{y\in\X}{R(y)=R(x)}\hspace{-1em}\e_yf(y)
    =\alpha\rho_E(f)\e_{R(x)}
    =0\e_{R(x)}
    =\boldsymbol{0}.
  \end{equation*}
  Therefore, $\alpha\e_x=\boldsymbol{0}$ holds for any $x\in\X$ and hence $\alpha=0$.
\end{proof}

\begin{proposition}\label{3.1.3}
  The representation~$\rho$ is covariant.
\end{proposition}
\begin{proof}
  In this proof, let $\I_E$ denote the covariant ideal of $E$.
  By Proposition~\ref{2.1.3.2}, for any $a\in\I_E$ and $f\in E$ we have the equation 
  \begin{equation*}
    \psi_\rho(\varphi(a))\rho_E(f)=\rho_A(a)\rho_E(f)
  \end{equation*}
  and hence 
  \begin{equation*}
    \bigl(\psi_\rho(\varphi(a))-\rho_A(a)\bigr)\rho_E(f)=0.
  \end{equation*}
  Applying Lemma~\ref{3.1.2}, we have $\psi_\rho(\varphi(a))-\rho_A(a)=0$, i.e., $\psi_\rho(\varphi(a))=\rho_A(a)$.
  Since this holds for any $a\in\I_E$, the representation $\rho$ is covariant.
\end{proof}

\subsection{$\rho_n$ and $E_n$}

For each $n\in\N\sqcup\{0\}$, we define $\rho_n$ and $E_n$ as below.
\begin{itemize}
  
  \item $\rho_n\colon C_0(X)\to M_\X$ is the complex-linear operator defined by
  \begin{equation*}
    \rho_n(f)\e_y:=\stacksum{x\in\X}{R^{\circ n}(x)=y}\hspace{-1em}\e_xf(x)\quad\text{for each $f\in E=C_0(X)$ and $y\in\X$.}
  \end{equation*}
  Endowing the domain with the supremum norm and the codomain with the operator norm, the operator $\rho_n$ is bi-Lipschitz.
  We note that $\rho_0=\rho_A$ and $\rho_1=\rho_E$.
  
  \item $E_n$ is the range of $\rho_n$. This is a Banach subspace of $M_\X$.
\end{itemize}
The next equation follows immediately from the definition of $\rho_n$:
\begin{equation*}
  \bigl(\text{the $(x,y)$-entry of $\rho_n(f)$}\bigr)
  =\inn{\e_x}{\rho_n(f)\e_y}
  =\begin{cases}
    f(x)&\text{if $R^{\circ n}(x)=y$},\\
    0&\text{otherwise}.
  \end{cases}
\end{equation*}

\begin{proposition}\label{3.2.1}
  The following equation holds for any $f,g\in C_0(X)$ and $m,n\in\N\sqcup\{0\}$:
  \begin{equation*}
    \rho_m(f)\rho_n(g)
    =\rho_{m+n}\bigl(f\cdot (g\circ R^{\circ m})\bigr).
  \end{equation*}
\end{proposition}
\begin{proof}
  We have the following for any $x,z\in\X$:
  \begin{align*}
    \inn{\e_x}{\rho_m(f)\rho_n(g)\e_z}
    &=\inn{\e_x}{\rho_m(f)\sum_{y\in\X}\e_y\inn{\e_y}{\rho_n(g)\e_z}}\\
    &=\sum_{y\in\X}\inn{\e_x}{\rho_m(f)\e_y\inn{\e_y}{\rho_n(g)\e_z}}\\
    &=\sum_{y\in\X}\inn{\e_x}{\rho_m(f)\e_y}\inn{\e_y}{\rho_n(g)\e_z}\\
    &=\inn{\e_x}{\rho_m(f)\e_{R^{\circ m}(x)}}\inn{\e_{R^{\circ m}(x)}}{\rho_n(g)\e_z}\\
    &=f(x)\inn{\e_{R^{\circ m}(x)}}{\rho_n(g)\e_z}\\
    &=\begin{cases}
      f(x)g(R^{\circ m}(x))&\text{if $R^{\circ (m+n)}(x)=z$}\\
      0&\text{otherwise}
    \end{cases}\\
    &=\inn{\e_x}{\rho_{m+n}\bigl(f\cdot (g\circ R^{\circ m})\bigr)\e_z}.
    \qedhere
  \end{align*}
\end{proof}

\begin{lemma}\label{3.2.2}
  The map $(f,g)\mapsto f\cdot (g\circ R^{\circ m})$ maps $C_c(X)\times C_c(X)$ onto $C_c(X)$.
\end{lemma}
\begin{proof}
  First, we show that the image of $C_c(X)\times C_c(X)$ is included in $C_c(X)$.
  The continuity of $f\cdot (g\circ R^{\circ m})$ follows immediately from the continuity of $f$ and $g$.
  The zeros of $f$ are also zeros of $f\cdot (g\circ R^{\circ m})$ and hence the support of $f$ includes the support of $f\cdot (g\circ R^{\circ m})$.
  Therefore, the compactness of the support of $f\cdot (g\circ R^{\circ m})$ follows from the compactness of the support of $f$.
  By the above, $f\cdot (g\circ R^{\circ m})$ is continuous and compactly supported for any $f,g\in C_c(X)$.
  
  Next, we show that the image of $C_c(X)\times C_c(X)$ includes $C_c(X)$.
  Take any $f\in C_c(X)$.
  The support of $f$ is compact and so is its image under $R^{\circ m}$.
  There exists a compactly supported continuous function $g$ which values $1$ on $R^{\circ m}(\supp f)$.
  The following holds for any $x\in X$:
  \begin{equation*}
    f(x)\cdot g(R^{\circ m}(x))
    =\begin{cases}
      f(x)\cdot 1&\text{if $x\in\supp f$}\\
      0\cdot g(R^{\circ m}(x))&\text{otherwise}
    \end{cases}
    =\begin{cases}
      f(x)&\text{if $x\in\supp f$}\\
      0&\text{otherwise}
    \end{cases}
    =f(x).
  \end{equation*}
  Therefore, $f=f\cdot (g\circ R^{\circ m})$ holds.
\end{proof}

\begin{lemma}\label{3.2.3}
  The map $(f,g)\mapsto f\cdot (g\circ R^{\circ m})$ maps $C_0(X)\times C_0(X)$ densely into $C_0(X)$.
\end{lemma}
\begin{proof}
  The map in the statement of Lemma~\ref{3.2.2} is a bounded bilinear map and hence has a unique continuous extension to $C_0(X)\times C_0(X)$.
  The range of the extension includes $C_c(X)$ and therefore dense in $C_0(X)$.
\end{proof}

\begin{proposition}\label{3.2.4}
  The equation $\overline{E_mE_n}=E_{m+n}$ holds for any $m,n\in\N\sqcup\{0\}$, where $E_mE_n$ denotes the set $\SET{}{\alpha\beta}{\text{$\alpha\in E_m$ and $\beta\in E_n$}}$.
\end{proposition}
\begin{proof}
  By Proposition~\ref{3.2.1}, Lemma~\ref{3.2.2} and Lemma~\ref{3.2.3}, we have the following, which imply $\overline{E_mE_n}=E_{m+n}$.
  \begin{equation*}
    E_mE_n
    =\rho_m(C_0(X))\rho_n(C_0(X))
    \subseteq\rho_{m+n}(C_0(X))
    =E_{m+n}.
  \end{equation*}
  \begin{gather*}
    \overline{E_mE_n}
    \supseteq\overline{\rho_m(C_c(X))\rho_n(C_c(X))}
    =\overline{\rho_{m+n}(C_c(X))}
    \\
    =\rho_{m+n}(\overline{C_c(X)})
    =\rho_{m+n}(C_0(X))
    =E_{m+n}.
    \qedhere
  \end{gather*}
\end{proof}

\begin{corollary}\label{3.2.5}
  $E_n$'s are included in $C^*(\rho)$.
\end{corollary}

\subsection{$X_{m,n}$ and $\X_{m,n}$}

For each $m,n\in\N\sqcup\{0\}$, we define sets $X_{m,n}$ and $\X_{m,n}$ by the following:
\begin{align*}
  X_{m,n}&:=\SET{}{(x,y)\in X\times X}{R^{\circ m}(x)=R^{\circ n}(y)},\\
  \X_{m,n}&:=(\X\times\X)\cap X_{m,n}.
\end{align*}

\begin{lemma}\label{3.3.1}
  Let
    $\Omega$ be a set,
    $R\colon\Omega\to\Omega$ be a map, and
    $X$ be a subset of~$\Omega$.
  Suppose that $X$ is completely invariant under~$R$.
  Then, the equation $R(X\cap U)=X\cap R(U)$ holds for any subset~$U$ of~$\Omega$.
\end{lemma}
\begin{proof}
  The inclusion~``$\subseteq$'' follows from
  \begin{equation*}
    R(X\cap U)
    \subseteq R(X)\cap R(U)
    =X\cap R(U).
  \end{equation*}
  We show the inverse inclusion.
  Take any element~$y$ in $X\cap R(U)$.
  There exists $x\in U$ such that $y=R(x)$.
  We have
  \begin{equation*}
    x\in R^{-1}(y)\subseteq R^{-1}(X)=X.
  \end{equation*}
  So $x$ is in $X\cap U$ and $y$ is in $R(X\cap U)$.
  Therefore, $X\cap R(U)\subseteq R(X\cap U)$ holds.
\end{proof}

\begin{proposition}\label{3.3.2}
  $R$ is open as a map from~$X$ to itself, i.e. the image of every open set of~$X$ under~$R$ is open in~$X$.
\end{proposition}
\begin{proof}
  We recall that $R$ is open as a map from the Riemann sphere to itself by Fact~\ref{2.2.1.3}.
  The open set of~$X$ can be described as $X\cap U$ for some open set~$U$ of the Riemann sphere.
  By Lemma~\ref{3.3.1}, its image under~$R$ coincides with $X\cap R(U)$, which is open in~$X$.
\end{proof}

\begin{proposition}\label{3.3.3}
  The spaces~$X_{m,n}$ are non-empty, locally compact, and perfect.
\end{proposition}
\begin{proof}
  We recall that $X$ is not empty.
  So we can take a point~$z$ in~$X$.
  Since $X$ is completely invariant under~$R$, there exists $x,y\in X$ such that $R^{\circ m}(x)=R^{\circ n}(y)=z$, which implies $(x,y)\in X_{m,n}$.
  Im particular, $X_{m,n}$ is not empty.
  
  Since $X_{m,n}$ coincides the preimage of the diagonal set of~$X\times X$ under the continuous map $(x,y)\mapsto (R^{\circ m}(x),R^{\circ n}(y))$, it is closed in~$X\times X$ and hence locally compact.
  
  We show the perfectness of~$X_{m,n}$.
  Take any point~$(x,y)$ in~$X_{m,n}$ and any open neighborhoods $U$ and $V$ of $x$ and $y$ in~$X$, respectively.
  Let $z:=R^{\circ m}(x)=R^{\circ n}(y)$.
  By Proposition~\ref{3.3.2}, $R^{\circ m}(U)$ and $R^{\circ n}(V)$ are open neighborhoods of~$z$ and so is their intersection~$W$.
  Since $X$ is perfect, $W$ has a point~$w$ distinct from~$z$.
  There exists $u\in U$ and $v\in V$ such that $w=R^{\circ m}(u)=R^{\circ n}(v)$, which implies $u\neq x$, $v\neq y$ and $(u,v)\in X_{m,n}$.
  Therefore, the basic open neighborhood $U\times V$ of~$(x,y)$ has the point~$(u,v)$ distinct from~$(x,y)$.
  By the above, every point~$(x,y)$ in~$X_{m,n}$ is not isolated and hence $X_{m,n}$ is perfect.
\end{proof}

\begin{proposition}\label{3.3.4}
  $\X_{m,n}$ is dense in~$X_{m,n}$ for any $m,n\in\N\sqcup\{0\}$.
\end{proposition}
\begin{proof}
  Let $m,n\in\N\sqcup\{0\}$.
  Take any point~$(x,y)$ in~$X_{m,n}$.
  Let $z:=R^{\circ m}(x)=R^{\circ n}(y)$.
  
  Take any open neighborhoods $U$ and $V$ of $x$ and $y$ in~$X$, respectively.
  By Proposition~\ref{3.3.2}, $R^{\circ m}(U)$ and $R^{\circ n}(V)$ are open neighborhoods of~$z$ and so is their intersection~$W$.
  By the denseness of $\X$, we can take a point~$w$ in~$\X\cap W$.
  There exist $u\in U$ and $v\in V$ such that $w=R^{\circ m}(u)=R^{\circ n}(v)$, which implies $(u,v)\in X_{m,n}$ and $u,v\in\X$.
  Therefore, the basic open neighborhood~$U\times V$ of~$(x,y)$ meets $\X_{m,n}$ at~$(u,v)$.

  By the above, every point~$(x,y)$ in~$X_{m,n}$ is an adherent point of~$\X_{m,n}$ and hence $\X_{m,n}$ is dense in~$X_{m,n}$.
\end{proof}

\begin{proposition}\label{3.3.5}
  Let $m,n,j,k\in\N\sqcup\{0\}$.
  
  (1) If $m-n\neq j-k$, then $\X_{m,n}$ and $\X_{j,k}$ are disjoint.
  
  (2) If $m-n=j-k$ and $m\leq j$, then $X_{m,n}\subseteq X_{j,k}$ and $\X_{m,n}\subseteq\X_{j,k}$ hold.
\end{proposition}
\begin{proof}\ \par
\noindent
(1)
Assume that $m-n\neq j-k$ and $\X_{m,n}\cap\X_{j,k}\neq\emptyset$.
Then, we can take a point~$(x,y)$ belonging to both $\X_{m,n}$ and $\X_{j,k}$.
We note that $x$ belongs to~$\X$.
\begin{gather*}
  R^{\circ (m+k)}(x)
  =R^{\circ k}(R^{\circ m}(x))
  =R^{\circ k}(R^{\circ n}(y))\\
  =R^{\circ n}(R^{\circ k}(y))
  =R^{\circ n}(R^{\circ j}(x))
  =R^{\circ (n+j)}(x).
\end{gather*}
The assumption implies $m+k\neq n+j$ and hence $x$ does not belong to~$\X$.
This leads to the contradiction.
Therefore, $m-n\neq j-k$ implies $\X_{m,n}\cap\X_{j,k}=\emptyset$.

\noindent
(2)
Assume that $m-n=j-k$.
Then, for any point~$(x,y)$ in~$X_{m,n}$, we have
\begin{equation*}
  R^{\circ j}(x)
  =R^{\circ (j-m)}(R^{\circ m}(x))
  =R^{\circ (j-m)}(R^{\circ n}(y))
  =R^{\circ (j-m+n)}(y)
  =R^{\circ k}(y),
\end{equation*}
and hence $(x,y)$ belongs to~$X_{j,k}$.
Therefore, $X_{m,n}\subseteq X_{j,k}$ holds.
Similarly, we can show that $\X_{m,n}\subseteq\X_{j,k}$.
\end{proof}

\subsection{$\rho_{m,n}$}\ \par
\begin{lemma}\label{3.4.1}
  Let $m,n\in\N$ with~$m\leq n$ and $\alpha$ be a complex $m\times n$ matrix.
  Then, the inequality \begin{equation*}(\text{the operator norm of $\alpha$})\leq n\times(\text{the max norm of $\alpha$})\end{equation*} holds.
  Here, the \emph{max norm} of a matrix is the maximum of the absolute values of entries of the matrix.
\end{lemma}
\begin{proof}
  An $m\times n$ matrix~$\alpha$ can be regarded as a $\Z/m\Z\times\Z/n\Z$ matrix.
  For each $k=1$, ..., $n$, let $\beta_k$ be the $\Z/m\Z\times\Z/n\Z$ matrix whose $(i,j)$-entry coincides with the $(i,j)$-entry of $\alpha$ if $i=j+k$ and otherwise vanishes.
  Each~$\beta_k$ is quasi-monomial and hence its operator norm is equal to its max norm, which is less than or equal to the max norm of~$\alpha$.
  Therefore, the operator norm of~$\alpha$, which can be described as the sum of~$\beta_k$'s, is less than or equal to~$n$ times by the max norm of~$\alpha$.
\end{proof}

\begin{proposition}\label{3.4.2}
  Let $m,n\in\N\sqcup\{0\}$ and let $f$ be a bounded function on~$\X_{m,n}$.
  Then, there exists a unique $\X$-squared matrix~$\alpha$ whose $(x,y)$-entry values $f(x,y)$ if $(x,y)\in\X_{m,n}$ and otherwise zero.
  Moreover, the following inequality holds:
  \begin{equation*}
    \sup f
    \leq (\text{the operator norm of $\alpha$})
    \leq d\times\sup f,
  \end{equation*}
  where we let $d:=\max(\deg R^{\circ m},\deg R^{\circ n})$.
\end{proposition}
In this proposition, $f$ need not be continuous.
\begin{proof}
  The uniqueness follows immediately from Proposition~\ref{2.1.4.1}.
  
  Let $\alpha$ be the the linear map from the algebraic linear span of the~$\e_x$ to~$\H_\X$ whose $(x,y)$-entry $\inn{\e_x}{\alpha\e_y}$ values $f(x,y)$ if $(x,y)\in\X_{m,n}$ and otherwise zero.
  We will show the boundedness of~$\alpha$.
  
  We recall that $\X$ is completely invariant under~$R$ and has no branch points of the iterations of~$R$.
  Therefore, the preimage of any point~$z\in\X$ under~$R^{\circ m}$, denoted by $R^{-m}(z)$, is included in~$\X$ and the number of its elements is $\deg R^{\circ m}$.
  The similar statement for~$n$ also holds.
  For each $z\in\X$, let $\alpha_z$ be the $R^{-m}(z)\times R^{-n}(z)$ matrix whose $(x,y)$-entry coincides with the $(x,y)$-entry of~$\alpha$.
  The size of~$\alpha_z$ is $\deg R^{\circ m}\times\deg R^{\circ n}$ and hence~$\alpha_z$ is a usual matrix.
  By Lemma~\ref{3.4.1}, the next inequality holds:
  \begin{equation*}
    (\text{the operator norm of $\alpha_z$})
    \leq d\times (\text{the max norm of $\alpha_z$})
    \leq d\times\sup f.
  \end{equation*}
  
  Take any finite linear combination $\boldsymbol{u}$ and $\boldsymbol{v}$ of the~$\e_x$.
  For each $z\in\X$, we define vectors $\boldsymbol{u}_z$ and $\boldsymbol{v}_z$ by
  \begin{equation*}
    \boldsymbol{u}_z
    :=\stacksum{x\in\X}{R^{\circ m}(x)=z}\hspace{-1em}\e_x\inn{\e_x}{\boldsymbol{u}}
    \quad\text{and}\quad
    \boldsymbol{v}_z
    :=\stacksum{y\in\X}{z=R^{\circ n}(y)}\hspace{-1em}\e_x\inn{\e_x}{\boldsymbol{v}}.
  \end{equation*}
  Then, we have the following.
  \begin{align*}
    \inn{\boldsymbol{u}}{\alpha\boldsymbol{v}}
    &=\sum_{x\in\X}\sum_{y\in\X}\inn{\boldsymbol{u}}{\e_x}\inn{\e_x}{\alpha\e_y}\inn{\e_y}{\boldsymbol{v}}\\
    &=\sum_{z\in\X}\hspace{-0.5em}\stacksum{(x,y)\in\X_{m,n}}{R^{\circ m}(x)=z=R^{\circ n}(y)}\hspace{-2em}\inn{\boldsymbol{u}}{\e_x}\inn{\e_x}{\alpha\e_y}\inn{\e_y}{\boldsymbol{v}}\\
    &=\sum_{z\in\X}\stacksum{x\in R^{-m}(z)}{y\in R^{-n}(z)}\hspace{-1em}\inn{\boldsymbol{u}_z}{\e_x}\inn{\e_x}{\alpha_z\e_y}\inn{\e_y}{\boldsymbol{v}_z}\\
    &=\sum_{z\in\X}\inn{\boldsymbol{u}_z}{\alpha_z\boldsymbol{v}_z}.
  \end{align*}
  \begin{gather*}
    \abs{\inn{\boldsymbol{u}}{\alpha\boldsymbol{v}}}
    \leq\sum_{z\in\X}\abs{\inn{\boldsymbol{u}_z}{\alpha_z\boldsymbol{v}_z}}
    \leq\sum_{z\in\X}\norm{\boldsymbol{u}_z}\norm{\alpha_z}\norm{\boldsymbol{v}_z}
    \leq d\times\sup f\times\sum_{z\in\X}\norm{\boldsymbol{u}_z}\norm{\boldsymbol{v}_z}
    \\
    \leq d\times\sup f\times
      \biggl(\sum_{z\in\X}\norm{\boldsymbol{u}_z}^2\biggr)^{1/2}
      \biggl(\sum_{z\in\X}\norm{\boldsymbol{v}_z}^2\biggr)^{1/2}
    =d\times\sup f\times\norm{\boldsymbol{u}}\norm{\boldsymbol{v}}.
  \end{gather*}
  Therefore, the linear operator~$\alpha$ is bounded and hence has a continuous extension to~$\H_\X$.
  The extension is an $\X$-squared matrix with operator norm $\leq d\times\sup f$ and its $(x,y)$-entry values $f(x,y)$ if $\X_{m,n}$ and otherwise zero.
  
  All that remains to show is that the operator norm of~$\alpha$ is greater than or equal to~$\sup f$.
  For any $(x,y)\in\X_{m,n}$, we have
  \begin{equation*}
    f(x,y)
    =\inn{\e_x}{\alpha\e_y}
    \leq\norm{\e_x}\norm{\alpha}\norm{\e_y}
    =\norm{\alpha}.
  \end{equation*}
  Therefore, the inequality $\sup f\leq\norm{\alpha}$ holds.
\end{proof}

Let $m,n\in\N\sqcup\{0\}$.
By Proposition~\ref{3.4.2}, we can define a bounded linear operator~$\rho_{m,n}$ from~$C_0(X_{m,n})$ to~$M_\X$ which maps $f$ to the $\X$-squared matrix whose $(x,y)$-entry values $f(x,y)$ if $(x,y)\in\X_{m,n}$ and otherwise zero.
This operator~$\rho_{m,n}$ is bi-Lipschitz.

\begin{proposition}\label{3.4.3}
  Let $f,g\in C_0(X)$.
  Then, the function on~$X_{m,n}$ which maps $(x,y)$ to $f(x)\overline{g(y)}$ belongs to~$C_0(X_{m,n})$.
\end{proposition}
\begin{proof}
  Consider the bounded bi-linear operator from~$C(X)\times C(X)$ to~$C(X_{m,n})$ which maps the pair $(f,g)$ to the function which maps $(x,y)$ to $f(x)\overline{g(y)}$.
  This operator maps the pair~$(f,g)$ of compactly supported functions to a function supported by the compact set~$\supp f\times\supp g$.
  Therefore, it maps $C_c(X)\times C_c(X)$ into~$C_c(X_{m,n})$ and hence $C_0(X)\times C_0(X)$ into~$C_0(X_{m,n})$.
\end{proof}

\begin{proposition}\label{3.4.4}
  Let $f,g\in C_0(X)$ and let $h$ be the function on~$X_{m,n}$ which maps $(x,y)$ to $f(x)\overline{g(y)}$.
  Then, the equation $\rho_m(f)\rho_n(g)^*=\rho_{m,n}(h)$ holds.
\end{proposition}
\begin{proof}
  We have the following for any $x,y\in\X$, which is followed by this proposition.
  \begin{align*}
    \inn{\e_x}{\rho_m(f)\rho_n(g)^*\e_y}
    &=\inn{\e_x}{\rho_m(f)\sum_{z\in\X}\e_z\inn{\e_z}{\rho_n(g)^*\e_y}}\\
    &=\sum_{z\in\X}\inn{\e_x}{\rho_m(f)\e_z\inn{\e_z}{\rho_n(g)^*\e_y}}\\
    &=\sum_{z\in\X}\inn{\e_x}{\rho_m(f)\e_z}\inn{\e_z}{\rho_n(g)^*\e_y}\\
    &=\sum_{z\in\X}\inn{\e_x}{\rho_m(f)\e_z}\overline{\inn{\e_y}{\rho_n(g)\e_z}}\\
    &=f(x)\overline{\inn{\e_y}{\rho_n(g)\e_{R^{\circ m}(x)}}}\\
    &=\begin{cases}
      f(x)\overline{g(y)}&\text{if $R^{\circ m}(x)=R^{\circ n}(y)$}\\
      0&\text{otherwise}
    \end{cases}\\
    &=\begin{cases}
      h(x,y)&\text{if $(x,y)\in\X_{m,n}$}\\
      0&\text{otherwise}
    \end{cases}\\
    &=\inn{\e_x}{\rho_{m,n}(h)\e_y}.
    \qedhere
  \end{align*}
\end{proof}

\subsection{$\U_n$ and $\U_{m,n}$}

For each $n\in\N\sqcup\{0\}$, let $\U_n$ be the family of all relatively compact open sets~$U$ of~$X$ on which $R^{\circ n}$ is homeomorphic.
This covers the complement of the branch set (i.e. the set of all branch points) of $R^{\circ n}$.
The family~$\U_n$ is closed under taking open subsets, that is, an open set included in some open set belonging to~$\U_n$ also belongs to~$\U_n$.

For each $m,n\in\N\sqcup\{0\}$, we define $\U_{m,n}$, $S_{m,n}$ and $C_c(\U_{m,n})$ by the following:
\begin{equation*}
  \U_{m,n}:=\SET\big{X_{m,n}\cap (U\times V)}{\text{$U\in\U_m$, $V\in\U_n$ and $R^{\circ m}(U)=R^{\circ n}(V)$}}.
\end{equation*}
\begin{equation*}
  S_{m,n}:=\SET\big{(x,y)\in X_{m,n}}{\text{$R^{\circ m}$ is branched at~$x$ or $R^{\circ n}$ is branched at~$y$}}.
\end{equation*}
\begin{equation*}
  C_c(\U_{m,n}):=\SET\big{f\in C_c(X_{m,n})}{\text{$\supp f\subseteq U$ holds for some $U\in\U_{m,n}$}}.
\end{equation*}
$\U_{m,n}$ is a family of open sets of~$X_{m,n}$.
Since a rational function has only finite many branch points, $S_{m,n}$ is finite and hence closed in~$X_{m,n}$.
So the complement of~$S_{m,n}$ is locally compact.

\begin{proposition}\label{3.5.1}
  The family~$\U_{m,n}$ covers the complement of~$S_{m,n}$ in~$X_{m,n}$.
\end{proposition}
\begin{proof}
  Take any point $(x,y)\in X_{m,n}\setminus S_{m,n}$.
  Then, $R^{\circ m}$ is not branched at~$x$ and hence there exists an open neighborhood~$U$ of~$x$ belonging to~$\U_m$.
  Similarly, there exists an open neighborhood~$V$ of~$y$ belonging to~$\U_n$.
  Since $R$ is an open map, the images $R^{\circ m}(U)$ and $R^{\circ n}(V)$ are open neighborhoods of $z:=R^{\circ m}(x)=R^{\circ n}(y)$ and so is their intersection, denoted by~$W$.
  Let $U'$ and $V'$ be the preimage of~$W$ under $R^{\circ m}$ and $R^{\circ n}$, respectively.
  Since $R^{\circ m}$ and $R^{\circ n}$ are respectively homeomorphic and in particular bijective on $U$ and $V$, the equation $R^{\circ m}(U)=W=R^{\circ n}(V)$ holds.
  Since $U$ belongs to~$\U_m$, so is its subset~$U'$.
  Similarly, since $V$ belongs to~$\U_n$, so is its subset~$V'$.
  By the above, $X_{m,n}\cap (U'\times V')$ belongs to~$\U_{m,n}$ and has the point~$(x,y)$.
\end{proof}

\begin{proposition}\label{3.5.2}
  Let $W\in\U_{m,n}$ and let $U$ and $V$ be the image of~$W$ under the first and second projection, respectively.
  Then, the following hold: $U\in\U_m$, $V\in\U_n$, and $R^{\circ m}(U)=R^{\circ n}(V)$.
  Moreover, the restriction of the first projection to~$W$ is a homeomorphism onto $U$, and similarly for~$V$.
\end{proposition}
\begin{proof}
  By the definition of~$\U_{m,n}$, $W$ can be described as $X_{m,n}\cap (U'\times V')$ for some $U'\in\U_m$ and $V'\in\U_n$ satisfying the condition $R^{\circ m}(U')=R^{\circ n}(V')=:W'$.
  By the definition of~$\U_m$, the restriction of~$R^{\circ m}$ to~$U'$ has the continuous inverse~$s\colon W'\to U'$.
  Similarly, the restriction of~$R^{\circ n}$ to~$V'$ has the continuous inverse $t\colon W'\to V'$.
  
  Now we claim that $W=\SET\big{(s(w),t(w))}{w\in W'}$, which implies $U=U'$ and $V=V'$.
  For any $w\in W'$, the point $(s(w),t(w))$ is in $W=X_{m,n}\cap (U'\times V')$ since $s(w)\in U'$, $t(w)\in V'$, and $R^{\circ m}(s(w))=w=R^{\circ n}(t(w))$; and hence the inclusion~``$\supseteq$'' holds.
  Take any point~$(x,y)\in W$.
  Since $W\subseteq X_{m,n}$, the equation $R^{\circ m}(x)=R^{\circ n}(y)=:w$ holds.
  The point~$w$ is in~$W'$.
  We have $x=s(R^{\circ m}(x))=s(w)$ and $y=t(R^{\circ n}(y))=t(w)$.
  Therefore, the inclusion~``$\subseteq$'' also holds.
  
  By the above, we obtain the following commutative diagram.
  \[
    \xymatrix{
      U' \ar@{=}[d] & W' \ar[l]_s \ar[r]^t \ar[d]|(.4){(s,t)} \ar@{}[dl]|\circlearrowright \ar@{}[dr]|\circlearrowright & V' \ar@{=}[d]\\
      U & W \ar[l]^{\mathrm{proj}_1} \ar[r]_{\mathrm{proj}_2} & V
    }
  \]
  Since $s$ and $(s,t)$ are homeomorphic, the restriction of the first projection to $W$ is a homeomorphism from~$W$ onto~$U$, and similarly for~$V$.
\end{proof}

\begin{proposition}\label{3.5.3}
  The algebraic linear span of~$C_c(\U_{m,n})$ coincides with
  \begin{equation*}
    C_c(X_{m,n}\setminus S_{m,n}):=\SET\big{f\in C_c(X_{m,n})}{\text{$f$ vanishes on $S_{m,n}$}}.
  \end{equation*}
\end{proposition}
\begin{proof}
  The span is clearly included in $C_c(X_{m,n}\setminus S_{m,n})$.
  Now we show the inverse inclusion.
  Take any $f\in C_c(X_{m,n}\setminus S_{m,n})$.
  $\U_{m,n}$ covers the compact set~$\supp f$ and hence has a finite subcover~$\{U_1, ..., U_k\}$.
  Let~$U_0$ be the complement of~$\supp f$.
  Then, $\{U_0, U_1, ..., U_k\}$ covers $X_{m,n}$.
  Let $\{a_0, a_1, ..., a_k\}$ be a partition of unity subordinated to~$\{U_0, U_1, ..., U_k\}$.
  $\supp a_0$ is disjoint from~$\supp f$ and hence $a_0f=0$.
  Therefore, we have $a_1f+\cdots+a_kf=f$.
  For each $i=1, ..., k$, the function~$a_if$ is continuous and $\supp a_if\subseteq\supp a_i\subseteq U_i$ holds.
  So $a_if$'s belong to~$C_c(\U_{m,n})$ and their sum~$f$ is belongs to its algebraic linear span.
\end{proof}

\begin{corollary}\label{3.5.4}
  The closed linear span of~$C_c(\U_{m,n})$ coincides with
  \begin{equation*}
    C_0(X_{m,n}\setminus S_{m,n}):=\SET\big{f\in C_0(X_{m,n})}{\text{$f$ vanishes on $S_{m,n}$}}.
  \end{equation*}
\end{corollary}

\subsection{$E_{m,n}$}

Let $m,n\in\N\sqcup\{0\}$.
We define $E_{m,n}$ as the range of~$\rho_{m,n}$.
Since $\rho_{m,n}$ is bi-Lipschitz, $E_{m,n}$ is complete and hence closed in~$M_\X$.
Therefore, $E_{m,n}$ is a Banach subspace of~$M_\X$.
We note that $E_{m,0}=E_m$.
In particular, $E_{0,0}=E_0=\rho_A(A)$ and $E_{1,0}=E_1=\rho_E(E)$ hold.

\begin{lemma}\label{3.6.1}
  $E_mE_n^*:=\SET{}{\alpha\beta^*}{\text{$\alpha\in E_m$ and $\beta\in E_n$}}$ is included in~$E_{m,n}$.
\end{lemma}
\begin{proof}
  It immediately follows from Proposition~\ref{3.4.3} and~\ref{3.4.4}.
\end{proof}

\begin{lemma}\label{3.6.2}
  The operator~$\rho_{m,n}$ maps $C_c(\U_{m,n})$ into~$E_mE_n^*$.
\end{lemma}
\begin{proof}
  Take any $f\in C_c(\U_{m,n})$.
  There exists $W\in\U_{m,n}$ with $\supp f\subseteq W$.
  Let $U$ and $V$ be the image of~$W$ under the first and second projection, respectively.
  By Proposition~\ref{3.5.2}, the restriction of the first projection to~$W$ is a homeomorphism onto~$U$ and hence has the continuous inverse, denoted by~$s$.
  Let $f_s$ be the continuous function on~$X$ which values $f(s(x))$ at $x\in U$ and otherwise zero.
  The image of~$\supp f$ under the second projection is compact and included in~$V$; hence there exists a continuous function~$g$ on~$X$ which values~$1$ on the image and whose support included in~$V$.
  For any $(x,y)\in X_{m,n}$, we have
  \begin{equation*}
    f_s(x)\overline{g(y)}
    =f(x,y)\cdot\overline{g(y)}
    =\begin{cases}
      f(x,y)\cdot\overline{1}&\text{if $(x,y)\in\supp f$}\\
      0\cdot\overline{g(y)}&\text{otherwise}
    \end{cases}
    =f(x,y).
  \end{equation*}
  Therefore, by Proposition~\ref{3.4.4}, $\rho_{m,n}(f)=\rho_m(f_s)\rho_n(g)^*\in E_mE_n^*$ holds.
\end{proof}

\begin{corollary}\label{3.6.3}
  The operator~$\rho_{m,n}$ maps $C_0(X_{m,n}\setminus S_{m,n})$ into the closed linear span of~$E_mE_n^*$.
\end{corollary}
\begin{proof}
  It immediately follows from Corollary~\ref{3.5.4} and Lemma~\ref{3.6.2}.
\end{proof}

\begin{proposition}\label{3.6.4}
  $E_{m,n}$ coincides with the closed linear span of~$E_mE_n^*$, and hence is included in~$C^*(\rho)$.
\end{proposition}
\begin{proof}
  By Lemma~\ref{3.6.1}, $E_{m,n}$ includes the closed linear span of~$E_mE_n^*$.
  Now we show the inverse inclusion.
  
  Take any $f\in C_0(X_{m,n})$.
  Let $(x_1,y_1)$, ..., $(x_k,y_k)$ be all the elements of the finite set~$S_{m,n}$.
  We can take relatively compact open neighborhoods~$U_i$ of~$x_i$ and $V_i$ of~$y_i$ in~$X$ for each~$i$ such that $U_i\times V_i$'s are pairwise disjoint.
  Also, for each~$i$, we can take
  \begin{itemize}
    \item a continuous function~$g_i$ on~$X$ supported by~$U_i$ which values $f(x_i,y_i)$ at~$x_i$ and vanishes at the other~$x_j$'s and
    \item a continuous function~$h_i$ on~$X$ supported by~$V_i$ which values~$1$ at~$y_i$ and vanishes at the other~$y_j$'s.
  \end{itemize}
  Let $f_i$ be the function on~$X_{m,n}$ which values $g_i(x)\overline{h_i(y)}$ at~$(x,y)$.
  By Proposition~\ref{3.4.4}, $\rho_{m,n}(f_i)=\rho_m(g_i)\rho_n(h_i)^*\in E_mE_n^*$ holds.
  Each~$f_i$ coincides with~$f$ at~$(x_i,y_i)$ and vanishes at the other~$(x_j,y_j)$'s.
  Hence, $f_0:=f-(f_1+\cdots+f_k)$ vanishes on~$S_{m,n}$ and belongs to $C_0(X_{m,n}\setminus S_{m,n})$.
  By Corollary~\ref{3.6.3}, $\rho_{m,n}(f_0)$ belongs to the closed linear span of~$E_mE_n^*$.
  
  By the above, $\rho_{m,n}(f)=\rho_{m,n}(f_0)+\rho_{m,n}(f_1)+\cdots+\rho_{m,n}(f_k)$ belongs to the closed linear span of~$E_mE_n^*$.
\end{proof}

\begin{proposition}\label{3.6.5}\ \par
  (1) $E_{0,0}$ coincides with $E_0=\rho_A(A)$.
  
  (2) $E_{m,0}$ coincides with the closed linear span of
  \begin{equation*}
    E_1^m:=\SET\big{\rho_E(f_1)\cdots\rho_E(f_m)}{f_1, ..., f_m\in E}.
  \end{equation*}
  
  (3) $E_{0,n}$ coincides with the closed linear span of
  \begin{equation*}
    E_1^{*n}:=\SET\big{\rho_E(f_1)^*\cdots\rho_E(f_n)^*}{f_1, ..., f_n\in E}.
  \end{equation*}
  
  (4) $E_{m,n}$ coincides with the closed linear span of
  \begin{equation*}
    E_1^mE_1^{*n}:=\SET\big{\alpha\beta}{\text{$\alpha\in E_1^m$ and $\beta\in E_1^{*n}$}}.
  \end{equation*}
\end{proposition}
\begin{proof}
  (1) follows from that $X_{0,0}$ is isomorphic to~$X$.
  (2) follows from Proposition~\ref{3.2.4} and the fact that $X_{m,0}$ is isomorphic to~$X$.
  (3) follows from (2).
  (4) follows from Proposition~\ref{3.2.4} and~\ref{3.6.4}.
\end{proof}

\subsection{$\widehat\rho$}

We define maps $\widehat\rho_A\colon A\to C^*(\rho)\otimes C^*(\Z)$ and $\widehat\rho_E\colon E\to C^*(\rho)\otimes C^*(\Z)$ by the following:
\begin{alignat*}{2}
  \widehat\rho_A(a)&:=\rho_A(a)\otimes\upsilon_0&\quad&\text{for $a\in A$},\\
  \widehat\rho_E(f)&:=\rho_E(f)\otimes\upsilon_1&\quad&\text{for $f\in E$}.
\end{alignat*}

\begin{proposition}
$\widehat\rho=(\widehat\rho_A,\widehat\rho_E)$ is a representation of the C*-correspondence~$E$ over~$A$.
\end{proposition}
\begin{proof}
It is easy to verify that $\widehat\rho_A$ is a *-homomorphism and that $\widehat\rho_E$ is a complex-linear map.

We show that $\widehat\rho$ preserves inner product.
For any $f,g\in E$, we have
\begin{gather*}
\widehat\rho_E(f)^*\widehat\rho_E(g)
=\bigl(\rho_E(f)\otimes\upsilon_1\bigr)^*\bigl(\rho_E(g)\otimes\upsilon_1\bigr)
=\bigl(\rho_E(f)^*\otimes\upsilon_{-1}\bigr)\bigl(\rho_E(g)\otimes\upsilon_1\bigr)
\\
=\bigl(\rho_E(f)^*\rho_E(g)\bigr)\otimes\bigl(\upsilon_{-1}\upsilon_1\bigr)
=\rho_A(\inn{f}{g})\otimes\upsilon_0
=\widehat\rho_A(\inn{f}{g}).
\end{gather*}
Therefore, $\widehat\rho$ preserves inner product.

We show that $\widehat\rho$ preserves left $A$-scalar multiplication.
For any $a\in A$ and $f\in E$, we have
\begin{gather*}
\widehat\rho_A(a)\widehat\rho_E(f)
=\bigl(\rho_A(a)\otimes\upsilon_0\bigr)\bigl(\rho_E(f)\otimes\upsilon_1\bigr)
\\
=\bigl(\rho_A(a)\rho_E(f)\bigr)\otimes\bigl(\upsilon_0\upsilon_1\bigr)
=\rho_E(af)\otimes\upsilon_1
=\widehat\rho_E(af).
\end{gather*}
Therefore, $\widehat\rho$ preserves left $A$-scalar multiplication.

By the above discussion, $\widehat\rho$ is a representation of $E$.
\end{proof}

\begin{lemma}\label{3.7.1}
  Let $\xi$ be an element of~$C^*(\rho)\otimes C^*(\Z)$.
  Assume that $\xi\widehat\rho_E(f)=0$ holds for any~$f\in E$.
  Then, $\xi=0$.
\end{lemma}
\begin{proof}
  By the assumption, we have the following for any~$n\in\Z$ and any~$f\in E$:
  \begin{equation*}
    (\id\otimes\delta_n)(\xi)\rho_E(f)
    =(\id\otimes\delta_{n+1})(\xi\widehat\rho_E(f))
    =(\id\otimes\delta_{n+1})(0)
    =0.
  \end{equation*}
  Applying Lemma~\ref{3.1.2}, we obtain that $(\id\otimes\delta_n)(\xi)=0$ for any~$n\in\Z$.
  Applying Proposition~\ref{2.1.5.1}, we obtain that~$\xi=0$.
\end{proof}

\begin{lemma}\label{3.7.2}
  The representation~$\widehat\rho$ is injective and covariant and also admits a gauge action.
\end{lemma}
\begin{proof}
  The injectivity of~$\widehat\rho$ follows from the injectivity of~$\rho$ (Proposition~\ref{3.1.1}).
  The covariantness follows from Proposition~\ref{2.1.3.2} and Lemma~\ref{3.7.1}.
  Recall the *-homomorphism~$\Phi_c$ in Proposition~\ref{2.1.5.2}.
  The operator $\id\otimes\Phi_c\colon C^*(\rho)\otimes C^*(\Z)\to C^*(\rho)\otimes C^*(\Z)$ is a gauge action for~$c$.
  Therefore, $\widehat\rho$ admits a gauge action.
\end{proof}

\begin{lemma}\label{3.7.3}
  The representation~$\widehat\rho$ is universal as a covariant representation and $C^*(\widehat\rho)$ is isomorphic to~$\O_R(X)$.
\end{lemma}
\begin{proof}
  By Lemma~\ref{3.7.2}, the representation~$\widehat\rho$ is injective and covariant and also admits a gauge action; 
  and hence by Fact~\ref{2.1.3.4}, the natural *-homomorphism $\O_R(X)\to C^*(\widehat\rho)$ is an isomorphism.
\end{proof}

\subsection{The universality of $\rho$}

By the universality of~$\widehat\rho$ (Lemma~\ref{3.7.3}), we obtain the natural $*$-homomorphism from~$C^*(\widehat\rho)$ to~$C^*(\rho)$, denoted by~$\nu$.

For each~$n\in\Z$, we define $\mathcal{E}_n$, $\mathcal{X}_n$, $F_b(\mathcal{X}_n)$, and~$\ent_n$ as below.
\begin{itemize}
  \item $\mathcal{E}_n$ is the closure of the range of~$\id\otimes\delta_n$ included in~$C^*(\rho)$. This is a Banach subspace of~$C^*(\rho)$.
  \item $\mathcal{X}_n$ is the union of~$\X_{i,j}$'s, where the indices $i$ and $j$ run over~$\N\sqcup\{0\}$ such that~$i-j=n$. By Proposition~\ref{3.3.5}, $\mathcal{X}_n$'s are pairwise disjoint.
  \item $F_b(\mathcal{X}_n)$ is the Banach space of all bounded functions on~$\mathcal{X}_n$ with supremum norm.
  \item $\ent_n$ is the map from~$C^*(\rho)$ to~$F_b(\mathcal{X}_n)$ which maps $\alpha$ to the function on~$\mathcal{X}_n$ which values the $(x,y)$-entry of~$\alpha$ at~$(x,y)\in\mathcal{X}_n$.
\end{itemize}

\begin{proposition}\label{3.8.1}
  $C^*(\rho)$ coincides with the closure of the additive subgroup of~$C^*(\rho)$ generated by the union of~$E_{m,n}$'s, where the indices $m$ and $n$ run over~$\N\sqcup\{0\}$.
\end{proposition}
\begin{proof}
  Since $E_1^*E_1=\rho_E(E)^*\rho_E(E)$ is included in $E_0=\rho_A(A)$, the union of $E_0$, ${E_1}^m$'s, $({E_1}^n)^*$'s and $({E_1}^m)({E_1}^n)^*$'s is closed under multiplication and involution.
  Therefore, the algebraic additive subgroup of~$C^*(\rho)$ generated by the union is a $*$-subalgebra of~$C^*(\rho)$ including $E_0\cup E_1=\rho_A(A)\cup\rho_E(E)$, and its closure coincides with~$C^*(\rho)$.
  By Proposition~\ref{3.6.4} and~\ref{3.6.5}, the closure of the algebraic additive subgroup of~$C^*(\rho)$ generated by the union of~$E_{m,n}$'s also coincides with~$C^*(\rho)$.
\end{proof}

\begin{proposition}\label{3.8.2}
  $C^*(\widehat\rho)$ coincides with the closure of the additive subgroup of~$C^*(\widehat\rho)$ generated by the union of~$E_{m,n}\otimes\upsilon_{m-n}$'s, where the indices $m$ and $n$ run over~$\N\sqcup\{0\}$.
\end{proposition}
\begin{proof}
  By Proposition~\ref{3.6.5}, $E_{m,n}\otimes\upsilon_{m-n}$'s are included in~$C^*(\widehat\rho)$.
  Therefore, we can apply the similar argument as in the proof of Proposition~\ref{3.8.1}.
\end{proof}

\begin{proposition}\label{3.8.3}
  Let $n\in\Z$.
  $\mathcal{E}_n$ is the closure of the additive subgroup of~$C^*(\rho)$ generated by the union of~$E_{i,j}$'s, where the indices $i$ and $j$ run over~$\N\sqcup\{0\}$ such that~$i-j=n$.
\end{proposition}
\begin{proof}
  The operator~$\id\otimes\delta_n$ maps any element $\alpha\otimes\upsilon_{i-j}\in E_{i,j}\otimes\upsilon_{i-j}$ to~$\alpha$ if~$i-j=n$ and to zero otherwise.
  Therefore, by Proposition~\ref{3.8.2}, $\mathcal{E}_n$ coincides with the closed linear span of~$E_{i,j}$'s, where the indices $i$ and $j$ run over~$\N\sqcup\{0\}$ such that~$i-j=n$.
  Since $E_{i,j}$'s are closed under scalar-multiplication, the span can be described as the closure of the additive subgroup of~$C^*(\rho)$ generated by such~$E_{i,j}$'s.
\end{proof}

\begin{proposition}\label{3.8.4}
  Let $n\in\Z$, $x,y\in\X$, and $\alpha\in\mathcal{E}_n$.
  Then, $x\rel\alpha y$ implies that $(x,y)\in\mathcal{X}_n$.
\end{proposition}
\begin{proof}
  We prove the contrapositive.
  Assume that $(x,y)\not\in\mathcal{X}_n$.
  Then, the $(x,y)$-entry of an $\X$-squared matrix belonging to some~$E_{i,j}$ with~$i-j=n$ values zero.
  This means that the bounded linear functional on~$M_\X$ which maps~$\alpha$ to its $(x,y)$-entry vanishes on the dense subset of~$\mathcal{E}_n$; hence the functional vanishes on the whole of~$\mathcal{E}_n$.
  In other words, $x\rel\alpha y$ does not hold for any $\alpha\in\mathcal{E}_n$.
\end{proof}

\begin{proposition}\label{3.8.5}
  Let $n\in\Z$.
  The restriction of~$\ent_n$ to~$\mathcal{E}_n$ is injective.
\end{proposition}
\begin{proof}
  Take any $\X$-squared matrix~$\alpha$ belonging to the kernel of~$\ent_n$.
  Since $\ent_n(\alpha)=0$, the $(x,y)$-entry of~$\alpha$ vanishes for any~$(x,y)\in\mathcal{X}_n$.
  By Proposition~\ref{3.8.4}, the other entries also vanishes.
  Therefore, all entries of~$\alpha$ values zero.
  By Proposition~\ref{2.1.4.1}, such matrix is only zero.
\end{proof}

\begin{proposition}\label{3.8.6}
  Let $n\in\Z$.
  The following diagram is commutative.
  \[
    \xymatrix@M=8pt{
      C^*(\widehat\rho)
        \ar[r]^{\nu}
        \ar[d]_{\id\otimes\delta_n}
      & C^*(\rho)
        \ar[d]^{\mathrm{ent}_n}
      \\
      \mathcal{E}_n
        \ar@{>->}[r]_{\mathrm{ent}_n}
      & F_b(\mathcal{X}_n)
    }
  \]
\end{proposition}
\begin{proof}
  Let $i,j\in\N\sqcup\{0\}$ and $\alpha\in E_{i,j}$.
  For any~$(x,y)\in\mathcal{X}_n$, we have the following.
  \begin{equation*}
    \ent_n(\nu(\alpha\otimes\upsilon_{i-j}))(x,y)
    =\ent_n(\alpha)(x,y).
  \end{equation*}
  \begin{align*}
    \ent_n(\id\otimes\delta_n(\alpha\otimes\upsilon_{i-j}))(x,y)
    &=\ent_n(\id(\alpha)\cdot\delta_n(\upsilon_{i-j}))(x,y)\\
    &=\begin{cases}
      \ent_n(\alpha\cdot 1)(x,y)&\text{if $i-j=n$}\\
      \ent_n(\alpha\cdot 0)(x,y)&\text{otherwise}
    \end{cases}\\
    &=\begin{cases}
      \ent_n(\alpha)(x,y)&\text{if $i-j=n$}\\
      0&\text{otherwise}.
    \end{cases}
  \end{align*}
  By Proposition~\ref{3.3.5}, the $(x,y)$-entry of~$\alpha$ vanishes if~$i-j\neq n$.
  Therefore, we have the equation
  \begin{equation*}
    \ent_n(\nu(\alpha\otimes\upsilon_{i-j}))(x,y)
    =\ent_n(\id\otimes\delta_n(\alpha\otimes\upsilon_{i-j}))(x,y).
  \end{equation*}
  By the above, for each~$i,j\in\N\sqcup\{0\}$, the equation ${\ent_n}\circ\nu={\ent_n}\circ({\id}\otimes\delta_n)$ holds on~$E_{i,j}\otimes\upsilon_{i-j}$ and hence, by Proposition~\ref{3.8.2}, on the whole of~$C^*(\widehat\rho)$, too.
\end{proof}

\begin{lemma}\label{3.8.7}
  The $*$-homomorphism~$\nu$ is isometry.
\end{lemma}
\begin{proof}
  It suffices to show the injectivity of~$\nu$.
  Take any element~$\xi$ of the kernel of~$\nu$.
  By Proposition~\ref{3.8.6}, we have
  \begin{equation*}
    \ent_n(\id\otimes\delta_n(\xi))
    =\ent_n(\nu(\xi))
    =\ent_n(0)
    =0.
  \end{equation*}
  By Proposition~\ref{3.8.5}, we obtain that~$\id\otimes\delta_n(\xi)=0$.
  This holds for any~$n\in\Z$.
  Therefore, by applying Proposition~\ref{2.1.5.1}, we have~$\xi=0$.
\end{proof}

\begin{proposition}\label{3.8.8}
  The representation~$\rho$ is universal as a covariant representation and $C^*(\rho)$ is isomorphic to~$\O_R(X)$.
\end{proposition}
\begin{proof}
  It follows from Lemma~\ref{3.7.3} and~\ref{3.8.7}.
\end{proof}

\section{Is $C_0(X)$ a Cartan Subalgebra of $\O_R(X)$?}
The definition of a Cartan subalgebra of a C*-algebra was introduced in~\cite{R}.

\begin{definition}[\cite{R}, Definition 4.5]\label{4.1}
  Let $A$ be a sub C*-algebra of a C*-algebra~$B$.
  
  \hangindent=2.5em
  (1) Its \emph{normalizer} is the set
  \begin{equation*}
    N(A):=\SET{}{b\in B}{\text{$bAb^*\subseteq A$ and $b^*Ab\subseteq A$}}.
  \end{equation*}
  
  \hangindent=2.5em
  (2) One says that $A$ is \emph{regular} if its normalizer~$N(A)$ generates $B$ as a C*-algebra.
\end{definition}

\begin{definition}[\cite{R}, Definition 5.1]\label{4.2}
  We shall say that an abelian sub-C*-algebra~$A$ of a C*-algebra~$B$ is a \emph{Cartan subalgebra} if

  \hangindent=2.5em
  (1) $A$ contains an approximate unit of~$B$;

  \hangindent=2.5em
  (2) $A$ is maximal abelian;

  \hangindent=2.5em
  (3) $A$ is regular;

  \hangindent=2.5em
  (4) there exists a faithful conditional expectation~$P$ of~$B$ onto~$A$.
\end{definition}

In this section, the normalizer of~$\rho_A(A)$ in~$C^*(\rho)$ is simply denoted by~$N$. Let $C^*(N)$ denote the C*-subalgebra of~$C^*(\rho)$ generated by~$N$.

\begin{definition}\label{4.3}
  Let $m\in\N\sqcup\{0\}$ and $\alpha$ be an $\X$-squared matrix.
  The function on~$\X_{m,0}$ which sends $(x,y)$ to the $(x,y)$-entry of~$\alpha$ is called \emph{the $m$-th entry function for~$\alpha$}.
  Also, its continuous extension to~$X_{m,0}$, if exists, is called \emph{the $m$-th extended entry function}.
\end{definition}

\begin{proposition}\label{4.4}
  Let $m\in\N\sqcup\{0\}$.
  Every $\X$-squared matrix belonging to~$C^*(\rho)$ admits the $m$-th extended entry function, which vanishes at infinity.
\end{proposition}
\begin{proof}
  By Proposition~\ref{3.3.5}, the $m$-th entry function for an $\X$-squared matrix belonging to~$E_{i,j}$ is the restriction of a continuous function on~$X_{i,j}$ vanishing at infinity if~$i-j=m$ and the constant zero function otherwise.
  Therefore, in any case, it has a continuous extension to~$X_{m,0}$.
  Moreover, the extension vanishes at infinity.
\end{proof}

\begin{proposition}\label{4.5}
  $C^*(\rho)\cap D_\X=\rho_A(A)$.
\end{proposition}
\begin{proof}
  The inclusion~``$\supseteq$'' holds obviously.
  Now we show the inverse inclusion.
  Take any~$\alpha\in C^*(\rho)\cap D_\X$.
  Let $a$ be the $0$-th extended entry function for~$\alpha$, whose existence is guaranteed by Proposition~\ref{4.4}.
  Since $\alpha$ is diagonal, we have that~$\alpha=\rho_{0,0}(a)$ and hence it belongs to~$E_{0,0}=\rho_A(A)$.
\end{proof}

\begin{lemma}\label{4.6}
  An approximate unit of~$\rho_A(A)$ is also an approximate unit of~$C^*(\rho)$.
\end{lemma}
\begin{proof}
  Let $\{\alpha_\lambda\}$ be an approximate unit of~$A$.
  Since $\rho_A$ is isomorphism between $A$ and $\rho_A(A)$, letting $a_\lambda:=\rho_A^{-1}(\alpha_\lambda)$ for each~$\lambda$, the family~$\{a_\lambda\}$ is also an approximate unit for~$A$.
  
  First we show that $\lim_\lambda(\rho_E(f)\rho_A(a_\lambda))=\rho_E(f)$ for any $f\in E$.
  Using Proposition~\ref{2.1.3.5}, we obtain this as following:
  \begin{equation*}
    \lim_\lambda(\rho_E(f)\rho_A(a_\lambda))
    =\lim_\lambda\rho_E(fa_\lambda)
    =\rho_E(\lim_\lambda fa_\lambda)
    =\rho_E(f).
  \end{equation*}
  
  Next we show that $\lim_\lambda(\rho_A(a_\lambda)\rho_E(f))=\rho_E(f)$ for any~$f\in E$.
  Since $\rho_E$ is isometry, it suffices to show that $\lim_\lambda(a_\lambda f)=f$ in~$E$.
  Here we recall that the left $A$-scalar multiplication of~$E$ is defined by $(a\cdot f)(x):=a(x)f(x)$ for all~$x\in X$.
  Therefore, $a_\lambda f$ coincides the product of~$a$ and $f$ in~$A=C_0(X)$.
  We note that the norm of~$E=C_0(X)$ is equivalent to the supremum norm.
  The equation $\lim_\lambda(a_\lambda f)=f$ in~$A$ is followed by that $\lim_\lambda(a_\lambda f)=f$ in~$E$.
  
  By the above, we have $\lim_\lambda(\beta\alpha_\lambda)=\beta$ and $\lim_\lambda(\alpha_\lambda\beta)=\beta$ for any~$\beta\in E_1$.
  After applying involutions on both hand sides of the letter equation, we obtain that $\lim_\lambda(\beta^*\alpha_\lambda)=\beta^*$.
  These results are followed by that the equation $\lim_\lambda(\beta\alpha_\lambda)=\beta$ holds for every element~$\beta$ of the algebraic subalgebra~$B$ of~$C^*(\rho)$ generated by~$E_0\cup E_1\cup E_1^*$.
  $B$ is a $*$-subalgebra of~$C^*(\rho)$ including $E_0$ and $E_1$; and hence dense in~$C^*(\rho)$.
  
  Take any~$\beta\in C^*(\rho)$.
  Let $\varepsilon$ be a positive real number.
  We can take $\beta'\in B$ satisfying $\norm{\beta'-\beta}<\varepsilon$.
  Then, we have the following:
  \begin{align*}
    \norm{\beta\alpha_\lambda-\beta}
    &\leq\norm{(\beta-\beta')(\alpha_\lambda-1)}+\norm{\beta'\alpha_\lambda-\beta'}\\
    &=\norm{\beta-\beta'}\bigl(\norm{\alpha_\lambda}+\norm{1}\bigr)+\norm{\beta'\alpha_\lambda-\beta'}\\
    &<2\varepsilon+\norm{\beta'\alpha_\lambda-\beta'}.
  \end{align*}
  Taking limits on~$\lambda$, we get $\lim_\lambda\norm{\beta\alpha_\lambda-\beta}<2\varepsilon$.
  Since this holds for any~$\varepsilon>0$, we obtain that $\lim_\lambda\norm{\beta\alpha_\lambda-\beta}=0$.
  In other words, $\lim_\lambda(\beta\alpha_\lambda)=\beta$ holds.
  
  Therefore, $\{\alpha_\lambda\}$ is an approximate unit for~$C^*(\rho)$.
\end{proof}

\begin{lemma}[\cite{KW17}, Theorem 4.11]\label{4.7}
  $\rho_A(A)$ is a maximal abelian subalgebra of~$C^*(\rho)$.
\end{lemma}

We give another simple proof based on the representation~$\rho$ here.

\begin{proof}
  Take any~$\alpha\in C^*(\rho)\setminus\rho_A(A)$.
  By Proposition~\ref{4.5}, we have that~$\alpha\not\in D_\X$ and hence there exists distinct two elements~$x,z\in\X$ satisfying that~$\inn{\e_x}{\alpha\e_z}\neq 0$.
  We can choose~$b\in C_0(X)=A$ with~$b(x)\neq b(z)$.
  Let $\beta:=\rho_A(b)$.
  \begin{align*}
    \inn{\e_x}{(\alpha\beta-\beta\alpha)\e_z}
    &=\inn{\e_x}{\alpha\beta\e_z}-\inn{\e_x}{\beta\alpha\e_z}\\
    &=\inn{\e_x}{\alpha\sum_{y\in\X}\e_y\inn{\e_y}{\beta\e_z}}-\inn{\e_x}{\beta\sum_{y\in\X}\e_y\inn{\e_y}{\alpha\e_z}}\\
    &=\sum_{y\in\X}\inn{\e_x}{\alpha\e_y}\inn{\e_y}{\beta\e_z}-\sum_{y\in\X}\inn{\e_x}{\beta\e_y}\inn{\e_y}{\alpha\e_z}\\
    &=\inn{\e_x}{\alpha\e_z}\inn{\e_z}{\beta\e_z}-\inn{\e_x}{\beta\e_x}\inn{\e_x}{\alpha\e_z}\\
    &=\inn{\e_x}{\alpha\e_z}b(z)-b(x)\inn{\e_x}{\alpha\e_z}\\
    &=\inn{\e_x}{\alpha\e_z}(b(z)-b(x))\\
    &\neq 0
  \end{align*}
  By the above, we have~$\alpha\beta\neq\beta\alpha$.
  Therefore, $\rho_A(A)$ is a maximal abelian subalgebra of~$C^*(\rho)$.
\end{proof}

\begin{lemma}\label{4.9}
  The restriction to~$C^*(\rho)$ of the operator~$\delta$ defined in Proposition~\ref{2.1.2.6} is a faithful conditional expectation from~$C^*(\rho)$ onto~$\rho_A(A)$.
\end{lemma}
\begin{proof}
  Let $\ent_\Delta$ be the bounded linear operator from~$C^*(\rho)$ onto~$C_0(X_{0,0})$ which maps $\alpha$ to the $0$-th extended entry function for~$\alpha$, whose existence is guaranteed by Proposition~\ref{4.4}.
  Then, $\delta{\restriction_{C^*(\rho)}}=\rho_{0,0}\circ\ent_\Delta$ holds and hence the image of~$C^*(\rho)$ under~$\delta$ coincides with~$E_{0,0}=\rho_A(A)$.

  By Proposition~\ref{2.1.2.6}, the restriction~$\delta{\restriction_{C^*(\rho)}}$ is idempotent, positive, and $\rho_A(A)$-linear, and moreover it is faithful as a positive operator.
  Therefore we can apply Fact~\ref{2.1.2.5} to~$\delta{\restriction_{C^*(\rho)}}$; and it is a faithful conditional expectation.
\end{proof}

\begin{lemma}\label{4.10}
  $N=C^*(\rho)\cap QM_\X$.
\end{lemma}
\begin{proof}
  First we show the inclusion~``$\supseteq$''.
  Take any $\alpha\in C^*(\rho)\cap QM_\X$ and $\beta\in\rho_A(A)$.
  Take any two points~$w,z\in \X$ satisfying the condition~$w\rel{\alpha^*\beta\alpha}z$.
  This condition implies that~$w\rel{\alpha^*}\rel\beta\rel\alpha z$.
  That is, $w\rel{\alpha^*}x$, $x\rel\beta y$, and $y\rel\alpha z$ hold for some~$x,y\in\X$.
  $w\rel{\alpha^*}x$ is followed by~$x\rel\alpha w$.
  Since $\beta$ is diagonal, $x\rel\beta y$ implies that~$x=y$.
  Since $\alpha$ is quasi-monomial, $x\rel\alpha w$ and $x=y\rel\alpha z$ are followed by that~$w=z$.
  Therefore, the binary relation~$\rel{\alpha^*\beta\alpha}$ is diagonal and the matrix~$\alpha^*\beta\alpha$ is diagonal.
  Besides this matrix belongs to~$C^*(\rho)$.
  By Proposition~\ref{4.5}, the matrix~$\alpha^*\beta\alpha$ belongs to~$\rho_A(A)$.
  In the same way, we can show that the matrix~$\alpha\beta\alpha^*$ also belongs to~$\rho_A(A)$.
  Since these results hold for any~$\beta\in\rho_A(A)$, we obtain that~$\alpha\in N$.
  By the above discussion, $N$ includes $C^*(\rho)\cap QM_\X$.
  
  Next we show the inverse inclusion~``$\subseteq$''.
  Take any~$\alpha\in C^*(\rho)\setminus QM_\X$.
  There exist three points~$u,v,w\in\X$ satisfying the condition
  \begin{center}
  (1) $u\rel\alpha v$, $u\rel\alpha w$, and $v\neq w$
  \quad or\quad
  (2) $u\rel\alpha w$, $v\rel\alpha w$, and $u\neq v$.
  \end{center}
  Without loss of generality, we may suppose that the condition (1) holds.
  Let $f$ and $g$ be the functions on~$\X$ which maps $x$ to the $(x,v)$- and $(x,w)$-entry of~$\alpha$, respectively.
  They are $L^2$~functions on the measure space~$\X$ with counting measure and do not vanish at~$u$.
  Since $X$ is first-countable and locally compact, we can take a decreasing sequence $U_1\supseteq U_2\supseteq\cdots$ of relatively compact open neighborhoods of~$u$ satisfying that~$\bigcap_{k\in\N}U_k=\{u\}$.
  For each~$k$, we can take a $[0,1]$-valued continuous function~$h_k$ on~$X$ which values~$1$ at~$u$ and which supported by~$U_k$.
  The sequence~$\{\bar{f}h_kg\}_{k\in\N}$ is dominated by~$\abs{fg}$ and converges pointwise to the function which values~$1$ at~$u$ and vanishes otherwise.
  Applying the dominated convergence theorem, we obtain that
  \begin{equation*}
    \lim_{k\to\infty}\sum_{x\in\X}\overline{f(x)}h_k(x)g(x)
    =f(u)g(u)
    \neq 0.
  \end{equation*}
  Therefore, letting~$h:=h_K$ for some~$K$, we have that $\sum_{x\in\X}\overline{f(x)}h(x)g(x)\neq 0$.
  \begin{align*}
    \inn{\e_v}{\alpha^*\rho_A(h)\alpha\e_w}
    &=\inn{\e_v}{\alpha^*\sum_{x\in\X}\e_x\inn{\e_x}{\rho_A(h)\sum_{y\in\X}\e_y\inn{\e_y}{\alpha\e_w}}}\\
    &=\sum_{x\in\X}\sum_{y\in\X}\inn{\e_v}{\alpha^*\e_x}\inn{\e_x}{\rho_A(h)\e_y}\inn{\e_y}{\alpha\e_w}\\
    &=\sum_{x\in\X}\sum_{y\in\X}\overline{\inn{\e_x}{\alpha\e_v}}\inn{\e_x}{\rho_A(h)\e_y}\inn{\e_y}{\alpha\e_w}\\
    &=\sum_{x\in\X}\overline{\inn{\e_x}{\alpha\e_v}}\inn{\e_x}{\rho_A(h)\e_x}\inn{\e_x}{\alpha\e_w}\\
    &=\sum_{x\in\X}\overline{f(x)}h(x)g(x)\\
    &\neq 0.
  \end{align*}
  Therefore, the $\X$-squared matrix~$\alpha^*\rho_A(h)\alpha$ is not diagonal and hence does not belong to~$\rho_A(A)$.
  From the existence of such~$h$, it follows that $\alpha^*\rho_A(A)\alpha\not\subseteq\rho_A(A)$ and hence that $\alpha\not\in N$.

  By the above discussion, we have that $C^*(\rho)\setminus QM_\X$ is included in $C^*(\rho)\setminus N$ and hence that $N\subseteq C^*(\rho)\cap QM_\X$.
\end{proof}

\begin{lemma}\label{4.11}
  Let $m\in\N\sqcup\{0\}$.
  The $m$-th extended entry function of an $\X$-squared matrix belonging to~$C^*(N)$ vanishes on~$S_{m,0}$.
\end{lemma}
\begin{proof}
  First we prove the weaker statement that the $m$-th extended entry function of an $\X$-squared matrix belonging to~$N$ vanishes on~$S_{m,0}$.
  To prove this statement, take~$\alpha\in N$ and a branch point~$t$ of~$R^{\circ m}$.
  Let $f$ be the $m$-th extended entry function for~$\alpha$.
  Now we show that $f(t,R^{\circ m}(t))=0$ by contradiction.
  Assume that $f(t,R^{\circ m}(t))\neq 0$.
  Let $U:=\SET{}{x\in X}{f(x,R^{\circ m}(x))\neq 0}$.
  This is an open neighborhood of~$t$.
  Since $R^{\circ m}$ branches at~$t$, there exist distinct two points~$u,v\in U$ such that $R^{\circ m}(u)=R^{\circ m}(v)=:w$.
  Since $\X$ is dense in~$X$, we may suppose that~$u\in\X$, from which it follows automatically that~$v,w\in\X$ and that $(u,w),(v,w)\in\X_{m,0}$.
  The relation~$u\rel\alpha w$ follows from
  \begin{gather*}
    \inn{\e_u}{\alpha\e_w}
    =f(u,w)
    =f(u,R^{\circ m}(u))
    \neq 0.
  \end{gather*}
  The relation~$v\rel\alpha w$ follows from the similar calculation.
  Therefore, $\alpha$ is not quasi-monomial.
  This contradicts Lemma~\ref{4.10}.
  So $f(t,R^{\circ m}(t))=0$ holds.
  By the above argument, the weaker statement is true.

  The normalizer~$N$ has the zero element of~$C^*(\rho)$ and is closed under multiplication, complex-scalar multiplication, and involution.
  Hence, the algebraic additive subgroup of~$C^*(\rho)$ generated by~$N$ is a $*$-subalgebra of~$C^*(\rho)$ and hence dense in~$C^*(N)$.
  Since the map from~$C^*(\rho)$ to~$C_0(X_{m,0})$ which maps $\alpha$ to the $m$-th extended entry function for~$\alpha$ is a bounded linear operator, the $m$-th extended entry function of an $\X$-squared matrix belonging to~$C^*(N)$ vanishes on~$S_{m,0}$.
\end{proof}

\begin{lemma}\label{4.12}
  If $X$ has a branch point of~$R$, $\rho_A(A)$ is not regular as a C*-subalgebra of~$C^*(\rho)$.
\end{lemma}
\begin{proof}
  Assume that $X$ has a branch point of~$R$.
  Then, the set~$S_{1.0}$ is not empty and we can take a compactly supported continuous function~$f$ on~$X_{1,0}$ which does not vanish on~$S_{1,0}$.
  The 1st extended function for~$\rho_{1,0}(f)$ is $f$ itself; and hence although $\rho_{1,0}(f)$ belongs to~$C^*(\rho)$, by Lemma~\ref{4.11}, it does not belong to~$C^*(N)$.
  Therefore, $N$ does not generate $C^*(\rho)$ and $\rho_A(A)$ is not regular as a C*-subalgebra of~$C^*(\rho)$.
\end{proof}

\begin{lemma}\label{4.13}
  If $X$ has no branch point of~$R$, $\rho_A(A)$ is regular as a C*-subalgebra of~$C^*(\rho)$.
\end{lemma}
\begin{proof}
  Assume that $X$ has no branch point of~$R$.
  Then, $X$ also has no branch points of iterations of~$R$ and hence all of~$S_{m,n}$'s are empty.
  By Proposition~\ref{3.5.2}, an $\X$-squared matrix belongs to~$\rho_{m,n}(C_c(\U_{m,n}))$ is quasi-monomial and hence, by Lemma~\ref{4.10}, $N$ includes $\rho_{m,n}(C_c(\U_{m,n}))$.
  By the bi-Lipschitzness of~$\rho_{m,n}$ and Corollary~\ref{3.5.4}, we have
  \begin{align*}
    \clspan\rho_{m,n}(C_c(\U_{m,n}))
    &=\rho_{m,n}(\clspan C_c(\U_{m,n}))\\
    &=\rho_{m,n}(C_0(X_{m,n}\setminus S_{m,n}))\\
    &=\rho_{m,n}(C_0(X_{m,n}))\\
    &=E_{m,n}.
  \end{align*}
  Therefore, $C^*(N)$ includes all of~$E_{m,n}$'s and hence coincides with~$C^*(\rho)$.
  This means that $\rho_A(A)$ is regular as a C*-subalgebra of~$C^*(\rho)$.
\end{proof}

\begin{theorem}\label{4.14}
  The following are equivalent.
  
  \hangindent=2.5em
  (1) $C_0(X)$ is a Cartan subalgebra of~$\O_R(X)$. That is, $\rho_A(A)$ is a Cartan subalgebra of~$C^*(\rho)$.
  
  (2) $X$ has no branch point of~$R$.
\end{theorem}
\begin{proof}
  We check the four conditions in Definition~\ref{4.2}.
  \begin{itemize}
    \item By Lemma~\ref{4.6}, $\rho_A(A)$ contains an approximate unit of~$C^*(\rho)$.
    \item By Lemma~\ref{4.7}, $\rho_A(A)$ is a maximal abelian subalgebra of~$C^*(\rho)$.
    \item By Lemma~\ref{4.12} and \ref{4.13}, $\rho_A(A)$ is regular if and only if $X$ has no branch point of~$R$.
    \item By Lemma~\ref{4.9}, there exists a faithful conditional expectation of~$\rho_A(A)$ onto~$C^*(\rho)$.
  \end{itemize}
  From these results, we conclude this theorem.
\end{proof}

\begin{example}\label{4.15}
  Consider the case of~$R(z)=z^2$.
  In this case, $R$ branches only at the point~$0$ and its Julia set~$J$ is the unit circle~$\abs{z}=1$ (Example~\ref{2.2.3.2}).
  $J$ has no branch points of~$R$.
  Therefore, $C_0(J)=C(J)$ is a Cartan subalgebra of~$\O_R(J)$.
\end{example}

\begin{example}\label{4.16}
  Consider the case of~$R(z)=z^2-2$.
  In this case, $R$ branches only at the point~$0$ and its Julia set~$J$ is the closed interval~$[-2,2]$ (Example~\ref{2.2.3.2}).
  $J$ has the branch point~$0$.
  Therefore, $C_0(J)=C(J)$ is not a Cartan subalgebra of~$\O_R(J)$.
\end{example}

\begin{remark}
  In the case of Example~\ref{4.16}, it is known that the Kajiwara--Watatani algebra is isomorphic to the Cuntz algebra $\O_\infty$ (See \cite{KW05} Example 4.2).
  So the representation $\rho$ is a universal representation of $\O_\infty$ constructed without quotients.
\end{remark}

\end{document}